\documentclass[11pt,reqno]{amsart}

\usepackage{hyperref}
\usepackage{graphicx}
\usepackage{color}
\usepackage{amsfonts,amssymb}
\usepackage{bbm} 

\usepackage{mathrsfs}

\usepackage{a4wide}

\newtheorem{neu}{}[section]
\newtheorem{Cor}[neu]{Corollary}
\newtheorem*{Cor*}{Corollary}
\newtheorem{Thm}[neu]{Theorem}
\newtheorem*{Thm*}{Theorem}
\newtheorem{Theorem}{Theorem}
\newtheorem{Prop}[neu]{Proposition}
\newtheorem*{Prop*}{Proposition}
\theoremstyle{definition}
\newtheorem{Lemma}[neu]{Lemma}
\newtheorem*{Rmk*}{Remark}
\newtheorem{Rmk}[neu]{Remark}

\newtheorem*{Ex*}{Example}

\newtheorem*{Qu*}{Question}

\newtheorem{Def}[neu]{Definition}
\newtheorem{Conv}[neu]{Convention}

\newcommand{\N}{\mathbb{N}}

\newcommand{\Z}{\mathbb{Z}}
\newcommand{\R}{\mathbb{R}}

\newcommand{\pf}{\longrightarrow}

\newcommand{\CZ}{\mu_{\mathrm{CZ}}}

\newcommand{\Morse}{\mu_{\mathrm{Morse}}}

\newcommand{\id}{\mathrm{id}}

\newcommand{\om}{\omega}

\newcommand{\A}{\mathcal{A}}

\renewcommand{\S}{\mathfrak{S}}

\newcommand{\D}{\mathbb{D}}

\newcommand{\M}{\mathcal{M}}
\newcommand{\Mh}{\widehat{\mathcal{M}}}
\newcommand{\Nm}{\mathcal{N}}

\newcommand{\B}{\mathcal{B}}

\renewcommand{\L}{\mathscr{L}}

\renewcommand{\H}{\mathrm{H}}

\newcommand{\Ham}{\mathrm{Ham}}

\newcommand{\RFH}{\mathrm{RFH}}
\newcommand{\RFC}{\mathrm{RFC}}

\newcommand{\CM}{\mathrm{CM}}
\newcommand{\HM}{\mathrm{HM}}

\newcommand{\Crit}{\mathrm{Crit}}

\newcommand{\beq}{\begin{equation}}
\newcommand{\beqn}{\begin{equation}\nonumber}
\newcommand{\eeq}{\end{equation}}

\newcommand{\bea}{\begin{equation}\begin{aligned}}
\newcommand{\bean}{\begin{equation}\begin{aligned}\nonumber}
\newcommand{\eea}{\end{aligned}\end{equation}}

\numberwithin{equation}{section}

\definecolor{Urs}{rgb}{0,.7,0}
\definecolor{Peter}{rgb}{0,0,1}
\definecolor{red}{rgb}{1,0,0}

\newcommand{\p}{\partial}
\newcommand{\Mp}{\mathfrak{M}}
\newcommand{\MP}{\mathcal{MP}}

\begin{document}
\title[Spectral Invariants in Rabinowitz Floer homology]{Spectral Invariants in Rabinowitz Floer homology and Global Hamiltonian perturbations}
\author{Peter Albers}
\author{Urs Frauenfelder}
\address{
    Peter Albers\\
    Department of Mathematics\\
    Purdue University}
\email{palbers@math.purdue.edu}
\address{
    Urs Frauenfelder\\
    Department of Mathematics and Research Institute of Mathematics\\
    Seoul National University}
\email{frauenf@snu.ac.kr}
\keywords{Leaf-wise Intersections, Rabinowitz Floer homology, Global Hamiltonian perturbations, Spectral invariants}
\subjclass[2000]{53D40, 37J10, 58J05}
\begin{abstract}
Spectral invariant were introduced in Hamiltonian Floer homology by Viterbo, Oh, and Schwarz. We extend this concept to Rabinowitz Floer homology. As an application we derive new quantitative existence results for leaf-wise intersections. The importance of spectral invariants for the presented application is that spectral invariants allow us to derive existence of critical points of the Rabinowitz action functional even in degenerate situations where the functional is not Morse.  
\end{abstract}
\maketitle

\section{Introduction}

We consider an autonomous Hamiltonian system $(M,\om,F)$ where $(M,\om)$ is a symplectic manifold and $F:M\pf\R$ is a smooth time-independent function. The dynamics is given by the flow $\phi_F^t$ of the Hamiltonian vector field $X_F$ which is defined implicitly by $\om(X_F,\cdot)=dF(\cdot)$. Since $F$ is autonomous the energy hypersurface $S=F^{-1}(0)$ is preserved under $\phi_F^t$. Therefore, $S$ is foliated by leaves $L_x:=\{\phi_F^t(x)\mid t\in\R\}$, $x\in S$.

It is a challenging problem to compare the system $F$ before and after a global perturbation occurring in the time interval $[0,1]$. Such a perturbation is described by a function $H:M\times[0,1]\pf\R$. J.~Moser observed in \cite{Moser_A_fixed_point_theorem_in_symplectic_geometry} that it is not possible to destroy all trajectories of the unperturbed system if the perturbation is sufficiently small, that is, there exists $x\in S$ 
\beq
\phi_H^1(x)\in L_x\;.
\eeq
Such a point $x$ is referred to as a leaf-wise intersection. Equivalently, there exists $(x,\eta)\in S\times\R$ such that
\beq
\phi_F^{\eta}(x)=\phi_{H}^1(x)\;.
\eeq
We point out that the time shift $\eta$ is uniquely defined by the above equation unless the leaf $L_x$ is closed. If the time shift is negative then the perturbation moves the system back into its own past. Likewise, if the time shift is positive the perturbation moves the system forward into its own future.

Already the existence problem for leaf-wise intersections is highly non-trivial. The search for leaf-wise intersections was initiated by Moser in \cite{Moser_A_fixed_point_theorem_in_symplectic_geometry} and pursued further in 
\cite{Banyaga_On_fixed_points_of_symplectic_maps,Hofer_On_the_topological_properties_of_symplectic_maps,
Ekeland_Hofer_Two_symplectic_fixed_point_theorems_with_applications_to_Hamiltonian_dynamics,Ginzburg_Coisotropic_intersections,Dragnev_Symplectic_rigidity_symplectic_fixed_points_and_global_perturbations_of_Hamiltonian_systems,Albers_Frauenfelder_Leafwise_intersections_and_RFH,Ziltener_coisotropic, Albers_Frauenfelder_Leafwise_Intersections_Are_Generically_Morse,Gurel_leafwise_coisotropic_intersection,Kang_Existence_of_leafwise_intersection_points_in_the_unrestricted_case,Merry_On_the_RFH_of_twisted_cotangent_bundles}. We refer to \cite{Albers_Frauenfelder_Leafwise_Intersections_Are_Generically_Morse} for a brief history.

To our knowledge the size of possible time shifts $\eta$ has not been studied so far.

\begin{Theorem}\label{thm:main_intro}
Let $B$ be a closed manifold with $\dim\H_*(\L_B)=\infty$ where $\L_B=C^{\infty}(S^1,B)$. Let  $(M:=T^*B,\om)$ be its cotangent bundle and $F:M\pf\R$ be a smooth function such that $S:=F^{-1}(0)$ is a regular level set which is  fiber-wise star-shaped. We assume that $H:M\times[0,1]\pf\R$ has compact support. Then there exist $(x,\eta)\in S\times\R$ such that
\beq
\phi_F^{\eta}(x)=\phi_{H}^1(x)
\eeq
with arbitrarily large positive and negative time shifts $\eta$.
\end{Theorem}

\begin{Rmk}
 Thus, in classical Hamiltonian dynamical systems perturbations can move the system arbitrarily far into the past and future.
\end{Rmk}

\begin{Rmk}
Theorem \ref{thm:main_intro} cannot be true for arbitrary energy surfaces $S$. Indeed if $S$ is Hamiltonianly displaceable there are no leaf-wise intersections at all for a displacing Hamiltonian $H$.
\end{Rmk}

\begin{Cor}
Under the assumptions of Theorem \ref{thm:main_intro} there exists infinitely many leaf-wise intersections or a leaf-wise intersection $x$ where $L_x$ is closed. The latter we refer to as periodic leaf-wise intersections.
\end{Cor}

We recall that if $\dim B\geq2$ generically there are no periodic leaf-wise intersection, therefore, generically there exist infinitely many leaf-wise intersections, see \cite{Albers_Frauenfelder_Leafwise_Intersections_Are_Generically_Morse}.

We use our variational approach to leaf-wise intersections by interpreting them as critical points of a perturbed Rabinowitz action functional, see \cite{Albers_Frauenfelder_Leafwise_intersections_and_RFH}. Rabinowitz Floer homology for unit cotangent bundle can be expressed with help of the homology $\H_*(\L_B)$ of the free loop space $\L_B$ of $B$, see \cite{Cieliebak_Frauenfelder_Oancea_Rabinowitz_Floer_homology_and_symplectic_homology,Abbondandolo_Schwarz_Estimates_and_computations_in_Rabinowitz_Floer_homology}. Hence, if the perturbed Rabinowitz action functional is Morse it has to have infinitely many critical points. The main difficulty in proving Theorem \ref{thm:main_intro} is to extend this result to degenerate situations in which Rabinowitz Floer homology cannot be directly defined. To overcome this problem we define spectral invariants for Rabinowitz Floer homology. Spectral invariants were introduced by Viterbo \cite{Viterbo_Symplectic_topology_as_the_geometry_of_generating_functions}, Oh \cite{Oh_Symplectic_topology_as_the_geometry_of_action_functional_I,Oh_Symplectic_topology_as_the_geometry_of_action_functional_II}, and Schwarz \cite{Schwarz_Action_spectrum} in the context of Hamiltonian Floer homology. An interesting and useful feature in Hamiltonian Floer theory is the relation between spectral invariants and the pair-of-pants product. This direction is not needed  for the applications in the present article and therefore not pursued. It is an interesting problem for the future to study product structures in Rabinowitz Floer homology and their relations to spectral invariants.

If the Rabinowitz functional is Morse the spectral invariants are defined by a standard minimax procedure. In order to extend them to arbitrary Rabinowitz action functionals one has to proof a local Lipschitz property. This is the main technical issue and occupies most of this article. Spectral invariants are useful since even in the degenerate case they assign critical values to a Rabinowitz Floer homology class.

\subsubsection*{Acknowledgments}
This article was written during visits of the first author at the Seoul National University  and the Institute for Advanced Study, Princeton and visits of the second author at the ETH Z\"urich and the Institute for Advanced Study, Princeton. Both authors thank these institutions for their stimulating working atmospheres. 

This material is based upon work supported by the National Science Foundation under agreement No.~DMS-0635607 and DMS-0903856. Any opinions, findings and conclusions or recommendations expressed in this material are those of the authors and do not necessarily reflect the views of the National Science Foundation.

\section{A variational approach to leaf-wise intersections}

We recall from \cite{Albers_Frauenfelder_Leafwise_Intersections_Are_Generically_Morse} the notion of Moser pair.

\begin{Def}\label{def:Moser_pair}
A pair $\Mp=(F,H)$ of Hamiltonian functions $F,H:M\times S^1\pf R$ is called a Moser pair if it satisfies
\beq
F(\cdot,t)=0\quad\forall t\in[\tfrac12,1]\qquad\text{and}\qquad H(\cdot,t)=0\quad\forall t\in[0,\tfrac12]\;,
\eeq
and $F$ is of the form $F(x,t)=\rho(t)f(x)$ for some smooth map $\rho:S^1\to [0,1]$ with $\int_0^1\rho(t) dt=1$ and $f:M\pf\R$. 
We denote the set of Moser pairs by $\mathcal{MP}(M)$.
\end{Def}

For a Moser pair $\Mp=(F,H)$ the perturbed Rabinowitz action functional is defined by
\bea
\A^\Mp:\L_M\times\R&\pf\R\\
(v,\eta)&\mapsto-\int_0^1v^*\lambda-\int_0^1H(v,t)dt-\eta\int_0^1F(v,t)dt
\eea
where $\L_M:=C^\infty(S^1,M)$. A critical point $(v,\eta)$ of $\A^\Mp$ is a solution of 
\beq\label{eqn:critical_points_eqn}
\left. 
\begin{aligned}
\partial_tv=\eta X_{F}(v,t)+X_H(v,t)\\
\int_0^1F(v,t)dt=0
\end{aligned}\right\}
\eeq
In his pioneering work  \cite{Rabinowitz_Periodic_solutions_of_Hamiltonian_systems} Rabinowitz studied the case of the unperturbed functional, that is, the case $H=0$. In this situation critical points correspond to closed characteristics on the energy hypersurface $F^{-1}(0)$.

In \cite{Albers_Frauenfelder_Leafwise_intersections_and_RFH} we observed that critical points of the perturbed Rabinowitz action functional $\A^\Mp$ give rise to leaf-wise intersections. 

\begin{Prop}[\cite{Albers_Frauenfelder_Leafwise_intersections_and_RFH}]\label{prop:critical_points_give_LI}
Let $(v,\eta)$ be a critical point of $\A^\Mp$ then $x:=v(\tfrac12)\in F^{-1}(0)$ and
\beq
\phi_H^1(x)\in L_x
\eeq
thus, $x$ is a leaf-wise intersection.
\end{Prop}

\section{Rabinowitz Floer homology}

Rabinowitz Floer homology is the semi-infinite Morse homology associated to the Rabinowitz action functional. In the unperturbed case it has been constructed in \cite{Cieliebak_Frauenfelder_Restrictions_to_displaceable_exact_contact_embeddings} under the assumption that the energy hypersurface $F^{-1}(0)$ is a smooth restricted contact-type hypersurface. This construction in the unperturbed case has been extended to stable hypersurfaces in \cite{Cieliebak_Frauenfelder_Paternain_Symplectic_Topology}. In \cite{Albers_Frauenfelder_Leafwise_intersections_and_RFH}  we extended the construction in the case of restricted contact-type hypersurface to the perturbed Rabinowitz action functionals. In this article we continue our study of the perturbed Rabinowitz action functional for restricted contact-type hypersurfaces.

Let $(W,\om=d\lambda)$ be a compact, exact symplectic manifold with contact type boundary $\Sigma=\p W$, that is, the Liouville vector field $L$ defined by $i_L\om=\lambda$ points outward along $\Sigma$. In particular, $(\Sigma,\alpha:=\lambda|_\Sigma)$ is contact. We denote by $M$ the completion of $W$ obtained by attaching the positive half of the symplectization of $\Sigma$, that is, $(M=W\cup_{\Sigma}(\Sigma\times\R_+), \om=d\lambda)$ where $\lambda$ is extended by $e^r\alpha$, $r\in\R_+$, over $\Sigma\times\R_+ $. Since $W$ is compact and exact the negative half $\Sigma\times\R_-$ of the symplectization embeds into $W$. In the following we will identify $\Sigma\times\R$ with its embedding into $M$.

We choose a smooth function $\rho:S^1=\R/\Z\to [0,1]$ with $\int_0^1\rho(t) dt=1$ and $\rho(t)=0$ for $t\in[\frac12,1]$. We fix $0<\delta<1$ once and for all  and choose a smooth monotone function $\beta:\R\pf\R$ with 
\beq
\beta(r)=
\begin{cases}
r&\text{for }|r|\leq\delta/2\\
\delta&\text{for }r\geq\delta\\
-\delta&\text{for }r\leq-\delta
\end{cases}
\eeq
For later convenience we require in addition that
\beq\label{eqn:derivative_of_beta_leq_2}
0\leq\beta'(s)\leq2\;.
\eeq
For any smooth function $f:\Sigma\pf\R$ we define
\beq\label{eqn:def_of_F_f}
F_f(y,t):=
\begin{cases}
\beta(r-f(x))\,\rho(t)&\text{for }y=(x,r)\in\Sigma\times\R\\
-\delta\,\rho(t)&\text{for }y\in M\setminus\big(\Sigma\times\R\big)
\end{cases}
\eeq
We denote by $\Sigma_f:=\{(x,f(x))\mid x\in\Sigma\}\subset M$ the graph of $f$ over $\Sigma$ and abbreviate $F:=F_0$.

\begin{Lemma}\label{lem:basic_lemma_in_contact_geometry}
The 1-form $\alpha_f:=\lambda|_{\Sigma_f}=e^f\alpha$ is a contact form on $\Sigma_f$ with Reeb vector field $R_f$ given by $X_{G_f}|_{\Sigma_f}$ where $X_{G_f}$ is the Hamiltonian vector field of the function $G_f(x,r):=r-f(x):\Sigma\times\R\pf\R$. In particular,
\beq
\lambda(X_{G_f})=1\;.
\eeq
\end{Lemma}

\begin{proof}
That $\alpha_f$ is a contact form is straight forward to check. In order to prove $R_f=X_{G_f}|_{\Sigma_f}$ we first note that $\Sigma_f=G_f^{-1}(0)$ and thus $X_{G_f}|_{\Sigma_f}$ is indeed tangent to $\Sigma_f$. It remains to check the following two equations on $\Sigma_f$
\begin{align}
\label{eqn:contact1}i_{X_{G_f}}d\alpha_f&=0,\\
\label{eqn:contact2}\alpha_f(X_{G_f})&=1\;.
\end{align}
The defining equation of $X_{G_f}$ is
\beq
i_{X_{G_f}}\Big(e^r(dr\wedge\alpha+d\alpha)\Big)=dG_f\;.
\eeq
On $\Sigma_f=\{r=f(x)\}$ this reads
\beq
i_{X_{G_f}}\Big(\underbrace{e^f(df\wedge\alpha+d\alpha)}_{=d\alpha_f}\Big)=dG_f|_{\Sigma_f}=0\;.
\eeq
This proves the equation \eqref{eqn:contact1}. To prove \eqref{eqn:contact2} we observe
\beq
1=dG_f \Big(\frac{\p}{\p r} \Big)=i_{\frac{\p}{\p r}}i_{X_{G_f}}\Big(e^r(dr\wedge\alpha+d\alpha)\Big)=e^rdr\Big(\frac{\p}{\p r}\Big)\alpha(X_{G_f})=e^r\alpha(X_{G_f})\;.
\eeq
On $\Sigma_f=\{r=f(x)\}$ this becomes
\beq
1=e^f \alpha(X_{G_f})=\alpha_f(X_{G_f})\;.
\eeq
\end{proof}

\begin{Def}\label{def:set_of_half_constant_Hamiltonians}
We set
\beq
\mathcal{H}:=\{H\in C^\infty(M\times S^1)\mid H\text{ has compact support and } H(t,\cdot)=0\quad\forall t\in[0,\tfrac12]\}
\eeq
\end{Def}

\begin{Rmk}
It's easy to see that the $\Ham(M,\om)\equiv\{\phi_H^1\mid H\in\mathcal{H}\}$, e.g.~\cite{Albers_Frauenfelder_Leafwise_intersections_and_RFH}, where $\phi_H^1$ is the time-1-map of the Hamiltonian flow of $H$.
\end{Rmk}

\begin{Def}\label{def:Moser_pair_Sigma}
We define the subset $\mathcal{MP}(\Sigma) $ of Moser pairs
\beq
\MP(\Sigma):=\{\Mp=(F_f,H)\mid f\in C^\infty(\Sigma),\;H\in\mathcal{H}\}
\eeq
where $F_f$ is defined in equation \eqref{eqn:def_of_F_f}. We call $\Mp\in\MP(\Sigma)$ a Moser pair adapted to $\Sigma$.
\end{Def}

Proposition \ref{prop:critical_points_give_LI} implies that for $\Mp\in\MP(\Sigma)$ critical points of the Rabinowitz action functional $\A^\Mp$ are leaf-wise intersections on $\Sigma_f$. We choose a compatible almost complex structure $\widetilde{J}$ on $M$ such that on a $\delta$-neighborhood of $\Sigma_f$ the almost complex structure is SFT-like with respect to the contact form $\alpha_f$, see \cite{BEHWZ_SFT_compactness}. That is, $\widetilde{J}$ interchanges the Reeb vector field $R_f$ and Liouville vector field $L$, preserves the contact distribution, and is translationally invariant. Here $\delta$ is the universally chosen constant, for instance as in the definition of $F_f$, see \eqref{eqn:def_of_F_f}. Now we change $\widetilde{J}$ to $J$ by requiring
\beq
JR_f=e^{r-f(x)}L,\quad JL=e^{-r+f(x)}R_f
\eeq
and that $J=\widetilde{J}$ on the contact distribution. Then $J$ still is a compatible almost complex structure. Such a $J$ is called twisted SFT-like.

\begin{Rmk}\label{rmk:norm_of_ReeB_Liouville_equal_1}
Since $J$ is twisted SFT-like we have on a $\delta$-neighborhood
\beq
||X_{G_f}||=||L||=1
\eeq
and since $\lambda(X_{G_f})=1$ 
\beq
||\lambda||=1\;.
\eeq
\end{Rmk}

Let $\Mp\in\MP(\Sigma)$ be an adapted Moser pair. The norm of the gradient of $\A^\Mp$ equals
\beq\label{eqn:norm_of_gradient}
||\nabla\A_H^F(u,\eta)||^2=||\partial_tu-X_{H}(u,t)-\eta X_F(u,t)||_{L^2}^2+\Big|\int_0^1F(u(t),t)dt\Big|^2
\eeq
where the $L^2$-norm is taken with respect to the metric $g(\cdot,\cdot):=\om(\cdot,J\cdot)$. We denote by $\L$ the component of the contractible loops in $M$.

\begin{Def}\label{def:gradient flow line}
A gradient flow line of $\A^\Mp$ is (formally) a map $w=(u,\eta)\in C^{\infty}(\R,\L\times\R)$
solving the ODE 
\beq\label{eqn:gradient flow line}
\partial_s w(s)+ \nabla\A^\Mp(w(s))=0\;,
\eeq
where the gradient is taken with respect to metric $\mathfrak{m}$ defined as follows. Let $(\hat{u}_1,\hat{\eta}_1)$ and $(\hat{u}_2,\hat{\eta}_2)$ be two tangent vectors in $T_{(u,\eta)}(\L\times\R)$. We set
\beq
\mathfrak{m}\big((\hat{u}_1,\hat{\eta}_1),\,(\hat{u}_2,\hat{\eta}_2)\big):=\int_0^1g\big(\hat{u}_1,\hat{u}_2\big)dt+\hat{\eta}_1\hat{\eta}_2\;.
\eeq
According to Floer's interpretation, \cite{Floer_The_unregularized_gradient_flow_of_the_symplectic_action}, this means that $u$ and $\eta$ are smooth maps $u:\R\times S^1\pf M$ and $\eta:\R\pf\R$ solving
\beq\label{eqn:gradient flow equation}\left.
\begin{aligned}
&\partial_su+J(u)\big(\partial_tu-X_{H}(u,t)-\eta X_F(u,t)\big)=0\\[1ex]
&\partial_s\eta-\int_0^1F_f(u,t)dt=0.
\end{aligned}
\;\;\right\}
\eeq
\end{Def}

\begin{Def}
A Moser pair $\Mp$ is called regular if $\A^\Mp$ is Morse.
\end{Def}

We recall the following 
\begin{Prop}[\cite{Albers_Frauenfelder_Leafwise_intersections_and_RFH}]\label{prop:generic_is_regular_general} 
A generic Moser pair is regular. 
\end{Prop}

We need the following slightly stronger version here.

\begin{Prop}[\cite{Albers_Frauenfelder_Leafwise_intersections_and_RFH}]\label{prop:generic_is_regular} 
A generic adapted Moser pair is regular (see Definition \ref{def:Moser_pair_Sigma}). 
\end{Prop}

\begin{proof}
We note that the property of $\A^{(F,H)}$ being Morse is in fact a property of the hypersurface $\Sigma=F^{-1}(0)$ as long as the defining function $F$ has $0$ as a regular value as is apparent from the proof of Proposition A.2 in \cite{Albers_Frauenfelder_Leafwise_intersections_and_RFH}. Moreover, the property of $\Sigma_f$ of being a graph is a $C^1$-open condition. Thus, the assertion follows from Proposition \ref{prop:generic_is_regular_general}.
\end{proof}

For a regular contact-type Moser pair $\Mp$ the Rabinowitz Floer homology $\RFH_*(\Mp)$ is defined from the following chain complex
\beq
\RFC_k(\Mp):=\Big\{\xi=\sum_{c\colon\CZ(c)=k}\xi_c\,c\mid\#\{c\in\Crit{\A^\Mp}\mid\xi_c\neq0\in\Z/2\text{ and }\A^\Mp(c)\geq\kappa\}<\infty\;\forall\kappa\in\R\Big\}
\eeq
where the boundary operator is defined by counting gradient flow lines of $\A^\Mp$ in the sense of Floer homology, see \cite{Cieliebak_Frauenfelder_Restrictions_to_displaceable_exact_contact_embeddings} for details. 

If the Moser pair is of the form $\Mp=(F_f,0)$ then $\A^\Mp$ is never Morse. But for a generic $F_f$ the action functional  $\A^\Mp$ is Morse-Bott with critical manifold being the disjoint union of constant solutions of the form $(p,0)$, $p\in\Sigma_f$, and a family of circles corresponding to closed characteristics of $\om$ on $\Sigma_f $.

\begin{Def}
An adapted Moser pair is called weakly regular if it is of the form just described or if it is regular. The set of adapted weakly regular Moser pairs is denoted by $\mathcal{MP}^{reg}(\Sigma)$.
\end{Def}

\begin{Rmk}\label{rmk:perturb_by_Morse_fctn}
For adapted weakly regular Moser pairs $\Mp$ Rabinowitz Floer homology $\RFH_*(\Mp)$ can still be defined by taking the critical points of a Morse function on the critical manifolds as generators, see \cite{Cieliebak_Frauenfelder_Restrictions_to_displaceable_exact_contact_embeddings} for details. 
\end{Rmk}

For $\Mp_0,\Mp_1\in\MP^{reg}(\Sigma)$ there exist canonical isomorphisms
\beq
\zeta_{\Mp_0}^{\Mp_1}:\RFH_*(\Mp_0)\pf\RFH_*(\Mp_1)
\eeq
called continuation homomorphisms. They satisfy
\beq
\zeta_{\Mp_1}^{\Mp_2}\circ\zeta_{\Mp_0}^{\Mp_1}=\zeta_{\Mp_0}^{\Mp_2},\qquad\zeta_{\Mp}^{\Mp}=\id_{\RFH(\Mp)}\;.
\eeq
We refer the reader to \cite{Cieliebak_Frauenfelder_Restrictions_to_displaceable_exact_contact_embeddings} for details.

\begin{Def} \label{def:zeta_RFH}
The inverse limit defined with respect to the continuation homomorphism is denoted by
\beq
\RFH_*\equiv\RFH_*(\Sigma,M):=\lim_{\longleftarrow}\RFH_*(\Mp).
\eeq
Moreover, we refer by
\beq
\zeta^{\Mp}:\RFH_*\pf\RFH_*(\Mp)
\eeq
to the canonical map which in our case is an isomorphism.
\end{Def}

\begin{Rmk}
The main difficulty in defining Floer homology is compactness up to breaking of gradient flow lines. The new obstacle in Rabinowitz Floer homology is to establish uniform $L^\infty$ bounds for the Lagrange multiplier $\eta(s)$ along gradient flow lines with fixed asymptotics. The crucial ingredient is a period-action inequality for almost critical points. This has been established in the current set-up in \cite[Lemma 2.11]{Albers_Frauenfelder_Leafwise_intersections_and_RFH}. In this article we present an enhanced version of this lemma, see Lemma \ref{lem:crucial_lemma_for_eta_bound}. This enhancement is needed to study continuity properties of spectral invariants in Rabinowitz Floer homology.
\end{Rmk}

We recall the definition of the cut-off function $\beta:\R\pf\R$
\beq
\beta(r)=
\begin{cases}
r&\text{for }|r|\leq\delta/2\\
\delta&\text{for }r\geq\delta\\
-\delta&\text{for }r\leq-\delta
\end{cases}
\eeq
and
\beq
F_f(y,t):=
\begin{cases}
\beta(r-f(x))\,\rho(t)&\text{for }y=(x,r)\in\Sigma\times\R\\
-\delta\,\rho(t)&\text{for }y\in M\setminus\big(\Sigma\times\R\big)
\end{cases}
\eeq

\begin{Def}\label{def:constant_kappa(H)}
We introduce a semi-norm on the set $\mathcal{H}$, see Definition \ref{def:set_of_half_constant_Hamiltonians}, by
\beq
\kappa(H):=\int_0^1\max\left|\lambda(x)[X_H(x,t)]-H(x,t)\right|dt\quad\forall H\in\mathcal{H}\;.
\eeq
\end{Def}

\begin{Lemma}\label{lem:crucial_lemma_for_eta_bound}
For all $(u,\eta)\in C^{\infty}(S^1,M)\times\R$ with
\beq
||\nabla\A^\Mp(u,\eta)||<\frac\delta4
\eeq
we have the estimate
\beq\label{eqn:eta_estimate}
|\eta|\leq \frac{2}{2-\delta}\Big(|\A^\Mp(u,\eta)|+\delta/4+\kappa(H)\Big)
\eeq
where the norm of the gradient is given in equation \eqref{eqn:norm_of_gradient}.
\end{Lemma}

\begin{Rmk}
We point out the constants appearing in Lemma \ref{lem:crucial_lemma_for_eta_bound} are independent of the function $f\in C^\infty(\Sigma)$ appearing in the Moser pair $\Mp=(F_f,H)$.
\end{Rmk}

\begin{proof}
We define
\beq
U_{\frac\delta2}(f):=\{(x,r)\mid x\in\Sigma,\; r\in(f(x)-\delta/2,f(x)+\delta/2)\}
\eeq
\underline{Claim 1}: Assume that $u(t)\in U_{\frac\delta2}(f)$ for all $t\in[0,\tfrac12]$, then
\beq
|\eta|\leq \frac{2}{2-\delta}\Big(|\A^\Mp(u,\eta)|+||\nabla\A^\Mp(u,\eta)||+\kappa(H)\Big)\;,
\eeq
where $\kappa(H)$ has been defined in Definition \ref{def:constant_kappa(H)}.

\begin{proof}[Proof of Claim 1]
We compute using Lemma \ref{lem:basic_lemma_in_contact_geometry}
\bean
|\A^\Mp(u,\eta)|&=\left|-\int_0^1u^*\lambda-\int_0^1H(t,u(t))dt-\eta\int_0^1F_f(t,u(t))dt\right|\\[1ex]
&=\left|-\int_0^1\lambda(u(t))\big[\partial_tu-X_{H}(t,u)-\eta X_{F_f}(t,u)\big]+\int_{\tfrac12}^1\lambda(u(t))\big[X_{H}(t,u)\big]dt\right.\\
&+\underbrace{\int_0^{\tfrac12}\lambda(u(t))\big[\eta \underbrace{X_{F_f}(t,u)}_{=\rho(t)X_{G_f}(u(t))}\big]dt}_{=\eta}-\int_0^1H(t,u(t))dt-\eta\int_0^1F_f(t,u(t))dt\Bigg|\\[1ex]
&=\left|\eta\Big(1-\int_0^1\underbrace{F_f(t,u(t))}_{\in(-\delta/2,\delta/2)}dt\Big)-\int_0^1\lambda(u(t))\big[\partial_tu-X_{H}(t,u)-\eta X_{F_f}(t,u)\big]\right.\\
&\left.+\int_{0}^1\Big[\lambda(u(t))\big[X_{H}(t,u)\big]-H(t,u(t))\Big]dt\right|\\[1ex]
&\geq\frac{|\eta|(2-\delta)}{2}-\underbrace{||\lambda_{|U_\delta(f)}||_{C^0}}_{=1}||\partial_tu-X_{H}(t,u)-\eta X_{F_f}(t,u)||_{L^1}-\kappa(H)\\[1ex]
&\geq\frac{|\eta|(2-\delta)}{2}-||\partial_tu-X_{H}(t,u)-\eta X_{F_f}(t,u)||_{L^2}-\kappa(H)\\[1ex]
&\geq\frac{|\eta|(2-\delta)}{2}-||\nabla\A^\Mp(u,\eta)||_{L^2}-\kappa(H)
\eea
where $||\lambda_{|U_\delta}||_{C^0}=1$  since $J$ is twisted SFT-like on $U_\delta(f)$. This inequality implies Claim 1.
\end{proof}

\noindent\underline{Claim 2}: If for $(u,\eta)$ there exists $t\in[0,\tfrac12]$ with $u(t)\not\in U_{\frac\delta2}(f)$ then $||\nabla_s\A^\Mp(u,\eta)||\geq\frac{\delta}{4}$.

\begin{proof}[Proof of Claim 2]
If in addition $u(t)\not\in U_{\frac\delta4}(f)$ holds for all $t\in[0,\tfrac12]$ then using \eqref{eqn:norm_of_gradient}
\beq
||\nabla\A^\Mp(u,\eta)||\geq\Big|\int_0^1F_f(u(t),t)dt\Big|\geq\frac{\delta}{4}\int_0^1\rho(t)dt=\frac{\delta}{4}\;.
\eeq
Otherwise there exists $t'\in[0,\tfrac12]$ with $u(t')\in U_{\frac\delta4}(f)$. Thus, we can find $0\leq a<b\leq\tfrac12 $ such that either
\beq
u(a)\in  \p U_{\frac\delta4}(f),\quad u(b)\in  \p U_{\frac\delta2}(f)\quad\text{and}\quad u(t)\in U_{\frac\delta2}(f)\setminus U_{\frac\delta4}(f)\;\forall t\in[a,b]
\eeq
or
\beq
u(a)\in  \p U_{\frac\delta2}(f),\quad u(b)\in  \p U_{\frac\delta4}(f)\quad\text{and}\quad u(t)\in U_{\frac\delta2}(f)\setminus U_{\frac\delta4}(f)\;\forall t\in[a,b]\;.
\eeq
We only treat the first case here. The second is completely analogous. We recall from Lemma \ref{lem:basic_lemma_in_contact_geometry} the definition $G_f(x,r)=r-f(x)$.
\bea
||\nabla\A^\Mp(u,\eta)||&\geq||\partial_tu-X_{H}(u,t)-\eta X_{F_f}(u,t)||_{L^2}\\[1ex]
&\geq\left(\int_a^b||\partial_tu-\underbrace{X_{H}(u,t)}_{=0}-\eta X_{F_f}(u,t)||^2dt\right)^{\tfrac12}\\[1ex]
&\geq\left(\int_a^b\frac{1}{||\nabla {G_f}||^2}\big|g_t(\partial_tu,\nabla {G_f})-\eta \underbrace{g(X_{F_f}(u,t),\nabla {G_f})}_{=0}\big|^2dt\right)^{\tfrac12}\\[1ex] 
&\geq\frac{1}{||\nabla {G_f}|_{U_{\frac\delta2}(f)}||_{C^0}}\left(\int_a^b\left|\frac{d}{dt}{G_f}(u(t))\right|^2dt\right)^{\tfrac12}\\[1ex]
&\geq\frac{1}{||\nabla {G_f}|_{U_{\frac\delta2}(f)}||_{C^0}}\int_a^b\left|\frac{d}{dt}{G_f}(u(t))\right|dt\\[1ex]
&\geq\frac{1}{||\nabla {G_f}|_{U_{\frac\delta2}(f)}||_{C^0}}\int_a^b\frac{d}{dt}{G_f}(u(t))dt\\[1ex]
&\geq\frac{\delta}{4||\nabla {G_f}|_{U_{\frac\delta2}(f)}||_{C^0}}\\
&=\frac{\delta}{4}
\eea
where we used $g(X_{F_f},\nabla {G_f})=d{G_f}(X_{F_f})=d{G_f}(\rho(t)X_{G_f})=0$ since on $U_{\frac\delta2}(f)$ it holds $F_f=\rho(t)G_f$. Moreover, according to Remark \ref{rmk:norm_of_ReeB_Liouville_equal_1} we have $||\nabla {G_f}||=||X_{G_f}||=1$ on $U_{\frac\delta2}(f)$. This proves Claim 2.
\end{proof}
To prove the Lemma we observe that the assumption $||\nabla\A^\Mp(u,\eta)||<\frac\delta4$ excludes the case treated in Claim 2. 
\end{proof}

\section{Warmup -- Spectral Invariants in Morse homology} \label{sec:warmup}

In this section we explain spectral invariants in the finite dimensional case. The main construction scheme is already visible in the finite dimensional, nevertheless, the proof of local Lipschitz continuity is much easier.

Let $M$ be a closed manifold and $f:M\pf\R$ a Morse function. We recall that the Morse chain complex $\CM_*(f)$ is the graded $\Z/2$ vector space generated by the set $\Crit (f)$ of critical points of $f$. The grading is given by the Morse index $\Morse$  of $f$. The boundary operator $\p:\CM_*(f)\pf\CM_{*-1}(f)$ is defined on generators by counting gradient flow lines. Indeed, we choose a Riemannian metric $g$ on $M$ such that stable and unstable manifold with respect to the negative gradient flow of $\nabla f=\nabla^gf$ intersect transversely, that is, $W^s(x)\pitchfork W^u(y)$ for all $x,y\in\Crit(f)$. Then the moduli space
\beq
\Mh(x_-,x_+):=\big\{\gamma:\R\pf M\mid\dot{\gamma}+\nabla f(\gamma)=0,\;\lim_{s\to\pm\infty}\gamma(s)=x_\pm\big\}
\eeq
is a smooth manifold of dimension $\dim\Mh(x_-,x_+)=\Morse(x_-)-\Morse(x_+)$. Moreover, $\R$ acts by shifting the $s$-coordinate and we denote the quotient by
\beq
\M(x_-,x_+):=\Mh(x_-,x_+)\big/\R\;.
\eeq
Moreover, if $\Morse(x_-)-\Morse(x_+)=1$ then $\M(x_-,x_+)$ is a finite set. We set 
\beq
m(x_-,x_+):=\#_2\M(x_-,x_+)
\eeq
the mod 2 number of elements in $\M(x_-,x_+)$. Then we can define the differential $\p=\p(f,g)$ as a linear map which is given on generators by
\beq
\p x_-:=\!\!\!\!\sum_{\substack{x_+\in\Crit(f)\\\Morse(x_-)-\Morse(x_+)=1}}\!\!\!\!m(x_-,x_+)\,x_+\;.
\eeq
It is a deep theorem in Morse homology that the identity
\beq
\p\circ\p=0
\eeq
holds, see \cite{Schwarz_Morse_homology} for details. Then
\beq
\HM_*(f,g):=\H_*(\CM_\bullet(f),\p(f,g))
\eeq
is the Morse homology of the pair $(f,g)$.

Up to canonical isomorphisms Morse homology does not depend on the Morse-Smale pair $(f,g)$. These canonical isomorphisms are called continuation homomorphisms and are constructed in the following way. For two Morse-Smale pairs $(f_\pm,g_\pm)$ we choose a $T>0$ and a smooth family $\{(f_s,g_s)\}_{s\in\R}$ of functions $f_s:M\pf\R$ and Riemannian metrics $g_s$ such that
\beq
f_s=\;\begin{cases}
f_-\text{ for } s\leq -T\\
f_+\text{ for } s\geq T
\end{cases}
\qquad
g_s=\;\begin{cases}
g_-\text{ for } s\leq -T\\
g_+\text{ for } s\geq T
\end{cases}
\eeq
For critical points $x_\pm\in\Crit(f_\pm)$ we consider the moduli spaces
\beq
\Nm(x_-,x_+)=\Nm(x_-,x_+;f_s,g_s):=\big\{\gamma:\R\pf M\mid\dot{\gamma}(s)+\nabla^{g_s} f_s\big(\gamma(s)\big)=0,\;\lim_{s\to\pm\infty}\gamma(s)=x_\pm\big\}\;.
\eeq
A homotopy $(f_s,g_s)$ is called regular if the moduli space $\Nm(x_-,x_+)$ is a smooth manifold of dimension $\dim\Nm(x_-,x_+)=\Morse(x_-)-\Morse(x_+)$. A generic homotopy is regular. Moreover, in the special case $f_s=f_-=f_+$ and $g_s=g_-=g_+$ we have the identity
\beq
\Nm(x_-,x_+)=\Mh(x_-,x_+)\;.
\eeq
If $\Morse(x_-)-\Morse(x_+)=0$ the space $\Nm(x_-,x_+)$ is compact and we set
\beq
n(x_-,x_+):=\#_2\Nm(x_-,x_+)\;.
\eeq
Then we can define a linear map
\bea
Z=Z(f_s,g_s):\CM_*(f_-)&\pf\CM_*(f_+)\\
x_-&\mapsto\!\!\!\!\sum_{\substack{x_+\in\Crit(f_+)\\\Morse(x_-)-\Morse(x_+)=0}}\!\!\!\!n(x_-,x_+)\,x_+\;.
\eea
We denote $\p_\pm:=\p(f_\pm,g_\pm)$. In the same manner as $\p\circ\p=0$ one proves in Morse homology
\beq
Z\circ\p_-=\p_+\circ Z\;,
\eeq
see \cite{Schwarz_Morse_homology}. In particular, on homology we obtain the map
\beq
\zeta:\HM_*(f_-,g_-)\pf\HM_*(f_+,g_+)\\
\eeq
which is the continuation homomorphism. By a homotopy-of-homotopies argument it is proved that $\zeta$ is independent of the chosen homotopy $(f_s,g_s)$, see \cite{Schwarz_Morse_homology}. Moreover, the continuation homomorphism is functorial in the following sense. If we fix three Morse-Smale pairs $(f_a,g_a)$, $(f_b,g_b)$, and $(f_c,g_c)$ we denote the corresponding continuation homomorphisms by $\zeta_a^b:\HM_*(f_a,g_a)\pf\HM_*(f_b,g_b)$ and similarly $\zeta_a^c$ and $\zeta_b^c$. Then we have the following identities
\beq
\zeta_a^c=\zeta^c_b\circ\zeta^b_a\quad\text{and}\quad\zeta^a_a=\id_{\HM_*(f_a,g_a)}\;.
\eeq
In particular, we conclude that $\zeta_a^b$ is an isomorphism with inverse $\zeta_b^a$.

\begin{Def}
Let $(f,g)$ be a Morse-Smale pair. For $\xi=\sum_x\xi_xx\neq0\in\CM_*(f)$ we set
\beq
f(\xi):=\max\{f(x)\mid \xi_x\neq0\}
\eeq
and for $X\neq0\in\HM_*(f,g)$ we set
\beq
\sigma(X):=\min\{f(\xi)\mid X=[\xi]\}\;.
\eeq
We call $\sigma(X)$ the spectral value of $X$. Thus, $\sigma$ is a map
\beq
\sigma:\bigcup_{(f,g)\text{ Morse-Smale}}\HM_*(f,g)\pf\R\;.
\eeq
\end{Def}

\begin{Thm}\label{thm:spectral_invariants_estimate_Morse}
Let $(f_\pm,g_\pm)$ be two Morse-Smale pairs. Let $X\neq0\in\HM_*(f_-,g_-)$ then
\beq
\min(f_+-f_-)\leq\sigma(\zeta(X))-\sigma(X)\leq\max(f_+-f_-)\;.
\eeq
\end{Thm}

\begin{Rmk}
The estimate in Theorem \ref{thm:spectral_invariants_estimate_Morse} is sharp as can be seen for example by choosing $f_+=f_-+$const.  
\end{Rmk}

An immediate corollary of Theorem \ref{thm:spectral_invariants_estimate_Morse} is the following.

\begin{Cor}\label{cor:spectral_invariants_do_not_depend_on_g_Morse}
The spectral invariant $\sigma(X)$ does not depend on the Riemannian metric $g$.
\end{Cor}

As preparation of the proof of Theorem \ref{thm:spectral_invariants_estimate_Morse} we consider the following special homotopy. We fix a smooth monotone function $\beta:\R\pf[0,1]$ satisfying $\beta(s)=0$ for $s\leq-T$ and $\beta(s)=1$ for $s\geq T$. Then we set
\beq
f_s:=\beta(s)f_++(1-\beta(s))f_-=\beta(s)(f_+-f_-)+f_-
\eeq
and choose any homotopy $g_s$ from $g_-$ to $g_+$. 
\begin{Lemma}\label{lem:action_estimate_continuation_Morse}
Let $(f_s,g_s)$ as above. If $\Nm(x_-,x_+;f_s,g_s)\neq\emptyset$ we have  
\beq
f_+(x_+)-f_-(x_-)\leq\max(f_+-f_-)
\eeq
\end{Lemma}

\begin{proof}
We choose an element $\gamma\in\Nm(x_-,x_+;f_s,g_s)$ and estimate
\bea
f_+(x_+)-f_-(x_-)&=\int_{-\infty}^\infty\frac{d}{ds}f_s(\gamma(s))ds\\
&=\int_{-\infty}^\infty \left\{df_s(\gamma(s))[\dot{\gamma}(s)]+\frac{\p f}{\p s}(\gamma(s))\right\}ds\\
&=\int_{-\infty}^\infty \left\{-df_s(\gamma(s))\Big[\nabla^{g_s}f_s\big(\gamma(s)\big)\Big]+\beta'(s)(f_+-f_-)(\gamma(s))\right\}ds\\
&=\int_{-\infty}^\infty \left\{\underbrace{-g_s(\gamma(s))\Big[\nabla^{g_s}f_s(\gamma(s)),\nabla^{g_s}f_s(\gamma(s))\Big]}_{\leq0}+\underbrace{\beta'(s)}_{\geq0}\underbrace{(f_+-f_-)(\gamma(s))}_{\leq\max(f_+-f_-)}\right\}ds\\
&\leq\max(f_+-f_-)\int_{-\infty}^\infty \beta'(s)ds\\
&=\max(f_+-f_-)\;.
\eea 
\end{proof}

\begin{Cor}\label{cor:action_estimate_continuation_Morse}
Let $X\neq0\in\HM_*(f_-,g_-)$, then
\beq
\sigma(\zeta(X))-\sigma(X)\leq\max(f_+-f_-)\;.
\eeq
\end{Cor}

\begin{proof}
We first assume that the homotopy $f_s=\beta(s)f_++(1-\beta(s))f_-$ is regular. Let $\xi=\sum_x\xi_xx\in\CM_*(f_-)$ be a representative of $X$. Then 
\beq
Z(\xi)=\sum_y\underbrace{\Big(\sum_{x} \xi_xn(x,y)\Big)}_{\eta_y}y
\eeq
and thus
\beq
f_+(Z(\xi))=\max\{f_+(y)\mid\eta_y\neq0\}\;.
\eeq
Now we choose $y\in\Crit(f_+)$ s.t.~$f_+(Z(\xi))=f_+(y)$. Since $\eta_y\neq0$ there exists $x\in\Crit(f_-)$ such that $\xi_xn(x,y)\neq0$, i.e.~$\xi_x\neq0$ and $n(x,y)\neq0$. In particular, $\Nm(x,y)\neq\emptyset$ and by Lemma \ref{lem:action_estimate_continuation_Morse} we conlude
\beq
f_+(y)-f_-(x)\leq\max(f_+-f_-)\;.
\eeq
Then using $\xi_x\neq0$ we estimate
\bea
f_+(Z(\xi))-f_-(\xi)\leq f_+(y)-f_-(x)\leq\max(f_+-f_-)
\eea
and
\bea
\sigma(\zeta(X))-\sigma(X)&=\min\{f_+(\eta)\mid[\eta]=\zeta(X)\}-\min\{f_-(\xi)\mid[\xi]=X\}\\
&\leq\min\{f_+(\eta)\mid \eta=Z(\xi),\;[\xi]=X\}-\min\{f_-(\xi)\mid[\xi]=X\}\\
&\leq\min\{f_+(Z(\xi))\mid [\xi]=X\}-\min\{f_-(\xi)\mid[\xi]=X\}\\
&\leq\min\{f_-(\xi)+\max(f_+-f_-)\mid [\xi]=X\}-\min\{f_-(\xi)\mid[\xi]=X\}\\
&=\max(f_+-f_-)\;.
\eea
If the homotopy $f_s$ from above is not regular then we can approximate it by regular homotopies. The Corollary follows by noting that the estimate of Lemma \ref{lem:action_estimate_continuation_Morse} is correct up to an arbitrarily small error.
\end{proof}

\begin{proof}[Proof of Theorem \ref{thm:spectral_invariants_estimate_Morse}]
According to Corollary \ref{cor:action_estimate_continuation_Morse} we have
\beq
\sigma(\zeta(X))-\sigma(X)\leq\max(f_+-f_-)\;.
\eeq
and thus by symmetry
\beq
\sigma(X)-\sigma(\zeta(X))=\sigma(\zeta^{-1}(\zeta(X)))-\sigma(\zeta(X))\leq\max(f_--f_+)=-\min(f_+-f_-)\;.
\eeq
\end{proof}

With help of the continuation homomorphism we define the inverse limit
\beq
\HM_*:=\lim_{\longleftarrow}\HM_*(f,g)\;.
\eeq
Thus, for any Morse-Smale pair $(f,g)$ we have an isomorphism
\beq
\zeta^{(f,g)}:\HM_*\pf\HM_*(f,g).
\eeq

\begin{Def}
For a Morse function $f$ and $Y\neq0\in\HM_*$ we set
\beq
\sigma_f(Y):=\sigma\Big(\zeta^{(f,g)}(Y)\Big)
\eeq
where $g$ is any Riemannian metric so that $(f,g)$ is Morse-Smale. Moreover, for fixed $Y\neq0\in\HM_*$ we define
\bea
\rho_Y:\{f\in C^\infty\mid f\text{ is Morse}\}&\pf\R\\
f&\mapsto \sigma_f(Y)\;.
\eea
\end{Def}

\begin{Rmk}
$\sigma_f$  is well-defined, see Corollary \ref{cor:spectral_invariants_do_not_depend_on_g_Morse}. 
\end{Rmk}

\begin{Cor}\label{cor:spectral_invariant_is_1_Lipschitz_in_f_Morse}
Let $Y\neq0\in\HM_*$ and $f_\pm$ be Morse functions then
\beq
|\rho_Y(f_+)-\rho_Y(f_-)|\leq||f_+-f_-||_{C^0(\Sigma)}:=\max_M |f_+-f_-|\;.
\eeq
That is, $\rho_Y$ is 1-Lipschitz continuous with respect to the $C^0$ norm.
\end{Cor}

\begin{proof}
We choose $g_\pm$ such that $(f_\pm,g_\pm)$ are Morse-Smale and set $X_\pm:=\zeta^{(f_\pm,g_\pm)}(Y)$. Then
\bea
|\rho_Y(f_+)-\rho_Y(f_-)|&=|\sigma(X_+)-\sigma(X_-)|\\
&\leq\max\{\max(f_--f_+),\;-\min(f_+-f_-)\}\\
&=||f_+-f_-||_{C^0(\Sigma)}
\eea
where in the inequality we use Theorem \ref{thm:spectral_invariants_estimate_Morse} and $X_+=\zeta(X_-)$.
\end{proof}

\begin{Def}
For $f\in C^1(M)$ we define the spectrum of $f$ 
\beq
\S(f):=f\big(\Crit(f)\big)
\eeq
to be the set of critical values of $f$.
\end{Def}

\begin{Cor}\label{cor:spectral_invariants_extend_to_C_0_Morse}
Let $Y\neq0\in\HM_*$. The map $\rho_Y$ has a unique extension to a 1-Lipschitz continuous function $\rho_Y:C^0(M)\pf\R$. Moreover, if $f\in C^1(M)$ then $\rho_Y(f)$ is in the spectrum of $f$:
\beq
\rho_Y(f)\in\S(f)\;.
\eeq
\end{Cor}

\begin{proof}
We recall that $\{f\in C^\infty\mid f\text{ is Morse}\}$ is dense in $C^0$. Therefore, for $f\in C^0(M)$ there exist Morse functions $f_n$ with $||f_n-f||_{C^0}\to0$. By Corollary \ref{cor:spectral_invariant_is_1_Lipschitz_in_f_Morse}
\beq
|\rho_Y(f_n)-\rho_Y(f_m)|\leq||f_n-f_m||_{C^0}
\eeq
the sequence $(\rho_Y(f_n))$ is a Cauchy sequence in $\R$, thus converges. We set (by abuse of notation)
\beq
\rho_Y(f):=\lim\rho_Y(f_n)
\eeq
and note that for two sequences $(f_n)$ and $(f'_n)$ with $||f_n-f||_{C^0},||f'_n-f||_{C^0}\to0$ we can again by Corollary  \ref{cor:spectral_invariant_is_1_Lipschitz_in_f_Morse} estimate
\beq
\lim|\rho_Y(f'_n)-\rho_Y(f_n)|\leq\lim||f'_n-f_n||_{C^0}=0\;.
\eeq
Thus, the extension $\rho_Y$ is well-defined. A similar argument shows that the extension $\rho_Y$ is 1-Lipschitz continuous.

In order to show that $\rho_Y(f)$ is a critical value of $f\in C^1(M)$ we first note that if $f$ is in addition Morse then $\rho_Y(f)$ is a critical value by the very definition of $\rho_Y$. For the general case we point out that the space of Morse functions is in fact dense in the space of $C^1$-functions. Thus, for $f\in C^1(M)$ we can find a sequence $f_n$ of smooth Morse functions such that $||f_n-f||_{C^1}\to0$. In particular, also $||f_n-f||_{C^0}\to0$ and therefore $\rho_Y(f)=\lim\rho_Y(f_n)$. Thus, there exists $x_n\in\Crit(f_n)$ such that $\rho_Y(f_n)=f_n(x_n)$. Because $M$ is compact we can choose a convergent subsequence $x_{n_\nu}\to x\in M$. Since $||f_n-f||_{C^1}\to0$ we conclude that $df(x)=\lim df_{n_\nu}(x_{n_\nu})=0$. Thus, $x\in\Crit f$. Finally,
\beq
f(x)=\lim f_{n_\nu}(x_{n_\nu})=\lim\rho_Y(f_{n_\nu})=\rho_Y(f)\;.
\eeq
This proves the claim.
\end{proof}

The following Theorem explains the term spectral invariant.

\begin{Thm}
Let $\{f_r\}_{r\in[0,1]}$ be a continuous family of $C^1$ functions such that the spectrum $\S(f_r)\subset\R$ is independent of $r$ and nowhere dense. Then
\beq
\rho_Y(f_0)=\rho_Y(f_1)
\eeq
for all $Y\neq0\in\HM_*$.
\end{Thm}

\begin{Rmk}
The assumption that $\S(f_r)$ is nowhere dense follows from Sard theorem if $f_r$ is sufficiently differentiable.
\end{Rmk}

\begin{proof}
We consider the function
\bea
{ }[0,1]&\pf\R\\
r&\mapsto \rho_Y(f_r)\;.
\eea 
By Corollary \ref{cor:spectral_invariants_extend_to_C_0_Morse} this map is continuous and by assumption takes values in a nowhere dense subset of $\R$. Thus, it's constant.
\end{proof}

\section{Spectral Invariants in Rabinowitz Floer homology}

\begin{Def}\label{def:spectral_invariants}
Let $\Mp\in\MP^{reg}(\Sigma)$. For $\xi=\sum_c\xi_c\,c\neq0\in\RFC_*(\Mp)$ we set
\beq
\A^\Mp(\xi):=\max\big\{\A^\Mp(c)\mid\xi_c\neq0\big\}
\eeq
and for $X\neq0\in\RFH_*(\Mp)$
\beq
\sigma_\Mp(X):=\inf\big\{\A^\Mp(\xi)\mid[\xi]=X\big\}\in\R\cup\{-\infty\}\;.
\eeq
We call $\sigma_\Mp(X)$ the spectral value of $X$.
\end{Def}

\begin{Rmk}
A priori the spectral value $\sigma_{\Mp}(X)$ depends on the almost complex structure $J$ used in the definition of the boundary operator in the Rabinowitz Floer complex. As in the warm-up (section \ref{sec:warmup}) it is easy to show that $\sigma_{\Mp}(X)$ is in fact independent of $J$.  
\end{Rmk}

\begin{Lemma}\label{lemma:finite_spectral_value_implies_critical_value}
Let $\Mp\in\MP^{reg}(\Sigma)$. If the spectral value satisfies $\sigma_\Mp(X)\in\R$ then it is a critical value:
\beq
\sigma_\Mp(X)\in\S(\A^\Mp):=\A^\Mp\big(\Crit(\A^\Mp)\big)\;.
\eeq
\end{Lemma}

\begin{proof}
Let $\xi_n\in\RFC_*(\Mp)$ be a sequence such that $X=[\xi_n]$ and 
\beq
\lim_{n\to\infty}\A^\Mp(\xi_n)=\sigma_\Mp(X)\;.
\eeq
By definition there exist $c_n=(u_n,\eta_n)\in\Crit(\A^\Mp)$ with the property
\beq
\A^\Mp(c_n)=\A^\Mp(\xi_n)\;.
\eeq
From Lemma \ref{lem:crucial_lemma_for_eta_bound} we conclude that there exists a constant $C=C(H)$ such that
\beq
|\eta_n|\leq C(|\A^\Mp(c_n)|+1)
\eeq
and since $\lim_{n\to\infty}\A^\Mp(\xi_n)=\sigma_\Mp(X)$ the Lagrange multipliers $\eta_n$ are uniformly bounded. Thus, by Arzela-Ascoli and the critical point equation \eqref{eqn:critical_points_eqn} there exists a convergent subsequence $c_{n_k}\to c^*\in\Crit(\A^\Mp)$ satisfying
\beq
\A^\Mp(c^*)=\sigma_\Mp(X)\;.
\eeq
\end{proof}

The goal of this section is to compare the spectral invariants for different Moser pairs. This is established in Theorem \ref{thm:main_action_estimate}. The main idea is to estimate how the action develops along the continuation homomorphisms. 

For that let $\Mp_\pm=(F_{f_\pm},H_\pm)\in\MP^{reg}(\Sigma)$. We abbreviate $\Sigma_\pm:=\Sigma_{f_\pm}$ and choose a smooth monotone function $\theta:\R\pf[0,1]$ with $\theta(s)=0$ for $s\leq0$ and $\theta(s)=1$ for $s\geq1$ with $0\leq\theta'(s)\leq2$. We set
\beq
f_s:=\theta(s)(f_+-f_-)+f_-
\eeq
and
\beq\label{eqn:def_of_F_s}
F_s:=F_{f_s}\quad\text{and}\quad H_s:=\theta(s)(H_+-H_-)+H_-\;.
\eeq
For the definition of the function $F_f$ we refer to equation \eqref{eqn:def_of_F_f}. We consider the following family of Rabinowitz action functionals
\beq\label{eqn:Rabinowitz action functional}
\A_s(u,\eta):=-\int_0^1u^*\lambda-\int_0^1H_s(u(t),t)dt-\eta\int_0^1F_s(u(t),t)dt
\eeq
and set
\beq
\A_\pm(u,\eta):=-\int_0^1u^*\lambda-\int_0^1H_\pm(u(t),t)dt-\eta\int_0^1F_{f_\pm}(u(t),t)dt\;.
\eeq
The continuation homomorphism $\zeta_{\Mp_-}^{\Mp_+}:\RFH_*(\Mp_-)\pf\RFH_*(\Mp_+)$ is defined by counting solutions of
\beq\label{eqn:s_dep_gradient_flow_eqn}
\partial_s w(s)+ \nabla\A_s(w(s))=0\;
\eeq
that is, $w=(u,\eta)$ solves the problem
\beq\label{eqn:s_dep_gradient_flow_equation}\left.
\begin{aligned}
&\partial_su+J(u)\big(\partial_tu-X_{H_s}(u,t)-\eta X_{F_s}(u,t)\big)=0\\[1ex]
&\partial_s\eta-\int_0^1F_s(u,t)dt=0.
\end{aligned}
\;\;\right\}
\eeq

\begin{Prop}\label{prop:main_action_estimate}
Let $w$ be a solution of $\eqref{eqn:s_dep_gradient_flow_eqn}$ with $\lim_{s\to\pm\infty}=w_\pm\in\Crit\A_\pm$. If
\beq
||f_+-f_-||_{C^0(\Sigma)}\leq\frac{\delta(2-\delta)}{128-56\delta}
\eeq
then 
\beq
\A_+(w_+)\leq\max\Big\{\Big(1+\frac{8\Delta_1}{2-\delta}\Big)\A_-(w_-),0\Big\}+\Delta_0+2\Delta_1\bigg(\frac{64-28\delta}{\delta(2-\delta)}\Delta_0+\frac{(\delta+4\Delta_2)}{2-\delta}\bigg)
\eeq
where we use the abbreviations
\beq
\Delta_0:=\int_0^1||H_-(\cdot,t)-H_+(\cdot,t)||_{C^0(M)}dt\;,\quad\Delta_1:=||f_+-f_-||_{C^0(\Sigma)},\quad\Delta_2:=\max\{\kappa(H_+),\kappa(H_-)\}\;.
\eeq
\end{Prop}

\begin{proof}
Since by definition
\beq
F_s(y,t)=
\begin{cases}
\rho(t)\beta\big(r-f_s(x)\big)&\text{for }y=(x,r)\in\Sigma\times\R\\
-\delta\,\rho(t)&\text{for }y\in M\setminus\big(\Sigma\times\R\big)
\end{cases}
\eeq
we have
\beq
F'_s(y,t):=\frac{d}{ds}F_s(y,t)=
\begin{cases}
-\rho(t)\beta'\big(r-f_s(x)\big)\theta'(s)\big(f_1(x)-f_0(x)\big)&\text{for }y=(x,r)\in\Sigma\times\R\\
0&\text{for }y\in M\setminus\big(\Sigma\times\R\big)
\end{cases}
\eeq
and thus
\beq
|F'_s(y,t)|\leq2\rho(t)\theta'(s)||f_+-f_-||_{C^0(\Sigma)}
\eeq
using that $0\leq\beta'\leq2$. For $v=(u,\eta)\in\L_M\times\R$ we compute
\beq
\A'_s(v):=\frac{\p\A_s}{\p s}(v)=-\int_0^1H'_s(u(t),t)dt-\eta\int_0^1F'_s(u(t),t)dt\;.
\eeq
We set
\beq
0\leq E^\sigma(w):=\int_{-\infty}^\sigma||\p_sw(s)||^2ds\;,\qquad 0\leq E_\sigma(w):=\int^{\infty}_\sigma||\p_sw(s)||^2ds\;.
\eeq
We recall that $w$ solves \eqref{eqn:s_dep_gradient_flow_eqn} and estimate
\bea
\A_\sigma(w(\sigma))&=\A_-(w_-)+\int_{-\infty}^\sigma\frac{d}{ds}\A_s(w(s))ds\\
&=\A_-(w_-)+\int_{-\infty}^\sigma\A'_s(w(s))+d\A_s(w(s))[\p_sw(s)]ds\\
&=\A_-(w_-)+\int_{-\infty}^\sigma\A'_s(w(s))+\langle\nabla\A_s(w(s)),\p_sw(s)\rangle ds\\
&=\A_-(w_-)+\int_{-\infty}^\sigma\A'_s(w(s))+\langle-\p_sw(s),\p_sw(s)\rangle ds\\
&=\A_-(w_-)+\int_{-\infty}^\sigma\A'_s(w(s))ds-E^\sigma(w)\\
&=\A_-(w_-)-E^\sigma(w)-\int_{-\infty}^\sigma\int_0^1 \left( H'_s(u(t),t)+\eta(s) F'_s(u(t),t)\right)dtds\\[2ex]
&\leq\A_-(w_-)-E^\sigma(w)\\
&\quad+\int_{-\infty}^\sigma\int_0^1\left(-\theta'(s)\min_M\{H_+(\cdot,t)-H_-(\cdot,t)\}+2\eta(s)\rho(t)\theta'(s)||f_+-f_-||_{C^0(\Sigma)}\right)dtds\\[2ex]
&=\A_-(w_-)-E^\sigma(w)\\
&\quad+\int_{-\infty}^\sigma\theta'(s)\int_0^1\left(-\min_M\{H_+(\cdot,t)-H_-(\cdot,t)\}+2\eta(s)\rho(t)||f_+-f_-||_{C^0(\Sigma)}\right)dtds\\[2ex]
&=\A_-(w_-)-E^\sigma(w)+\int_{-\infty}^\sigma\theta'(s)\big(||H_+-H_-||_-+2\eta(s)||f_+-f_-||_{C^0(\Sigma)}\big)ds\\[2ex]
&\leq\A_-(w_-)-E^\sigma(w)+||H_+-H_-||_-+2||\eta||_{C^0(\R)}||f_+-f_-||_{C^0(\Sigma)}
\eea
where
\beqn
||H_+-H_-||_+=\int_0^1\max_M\{(H_+-H_-)(\cdot,t)\}dt,\;||H_+-H_-||_-=-\int_0^1\min_M\{(H_+-H_-)(\cdot,t)\}dt
\eeq
and similarly it holds
\beq
\A_\sigma(w(\sigma))\geq\A_+(w_+)+E_\sigma(w)-||H_+-H_-||_+-2||\eta||_{C^0(\R)}||f_+-f_-||_{C^0(\Sigma)}\;.
\eeq

In particular, we have
\bea\label{est:absolute_value_of_action_fctl}
|\A_\sigma(w(\sigma))|&\leq\max\{\A_-(w_-),-\A_+(w_+)\}\\
&\quad+\int_0^1||H_-(\cdot,t)-H_+(\cdot,t)||_{C^0(M)}dt+2||\eta||_{C^0(\R)}||f_+-f_-||_{C^0(\Sigma)}\;.
\eea
Moreover, we obtain for $\sigma=+\infty$
\beq\label{eqn:inequality_1}
\A_+(w_+)\leq\A_-(w_-)-E(w)+||H_+-H_-||_-+2||\eta||_{C^0(\R)}||f_+-f_-||_{C^0(\Sigma)}\;.
\eeq
For $\sigma\in\R$ we define
\beq
\tau(\sigma):=\inf\Big\{\tau\geq0\mid||\nabla\A_{\sigma+\tau}(w(\sigma+\tau))||\leq\frac\delta4\Big\}
\eeq
where $\delta$ is as in Lemma \ref{lem:crucial_lemma_for_eta_bound}. Then we compute for the energy
\beq\label{est:tau}
E(w)=\int_{-\infty}^\infty||\p_sw(s)||^2ds\geq\int_{\sigma}^{\sigma+\tau(\sigma)}\underbrace{||\nabla\A_s(w(s))||^2}_{\geq\frac{\delta^2}{16}}ds\geq\tau(\sigma)\frac{\delta^2}{16}\;.
\eeq
Combining the estimates \eqref{eqn:inequality_1} and \eqref{est:tau} we obtain
\bea
\tau(\sigma)&\leq\frac{16}{\delta^2}E(w)\leq\frac{16}{\delta^2}\Big(\A_-(w_-)-\A_+(w+)+||H_+-H_-||_-+2||\eta||_{C^0(\R)}||f_+-f_-||_{C^0(\Sigma)}\Big)\;.
\eea
From the definition of $F_s$ (see equation \eqref{eqn:def_of_F_s}) it follows
\beq
\left|\int_0^1F_s(t,u(t))dt\right|\leq\delta\int_0^1\rho(t)dt=\delta\;.
\eeq
From these estimates and the gradient flow equation \eqref{eqn:s_dep_gradient_flow_equation}
\beq
\p_s\eta(s)=\int_0^1F_s(t,u(t))dt
\eeq
we obtain
\bea
|\eta(\sigma)|&\leq|\eta(\sigma+\tau(\sigma))|+\int_\sigma^{\tau(\sigma)+\sigma}|\p_s\eta(s)|ds\\
&=|\eta(\sigma+\tau(\sigma))|+\int_\sigma^{\tau(\sigma)+\sigma}\left|\int_0^1F_s(t,u(t))dt\right|ds\\
&\leq|\eta(\sigma+\tau(\sigma))|+\tau(\sigma)\delta\\
&\leq|\eta(\sigma+\tau(\sigma))|\\
&\quad+\frac{16}{\delta}\Big(\A_-(w_-)-\A_+(w+)+||H_+-H_-||_-+2||\eta||_{C^0(\R)}||f_+-f_-||_{C^0(\Sigma)}\Big)
\eea
Using Lemma \ref{lem:crucial_lemma_for_eta_bound}, the definition of $\tau(\sigma)$, and estimate \eqref{est:absolute_value_of_action_fctl} we get
\bea
|\eta(\sigma+\tau(\sigma))|&\leq \frac{2}{2-\delta}\Big(|\A_\sigma(w(\sigma))|+\delta/4+\max\{\kappa(H_+),\kappa(H_-)\}\Big)\\
&\leq \frac{2}{2-\delta}\bigg(\max\{\A_-(w_-),-\A_+(w_+)\}+\int_0^1||H_-(\cdot,t)-H_+(\cdot,t)||_{C^0(M)}dt\\
&\qquad\qquad\qquad+2||\eta||_{C^0(\R)}||f_+-f_-||_{C^0(\Sigma)}+\delta/4+\kappa(H)\bigg)
\eea
where we used that $\kappa(H)$ is a semi-norm, in particular,
\bea
\kappa(H_s)&=\kappa\big(\theta(s)H_++(1-\theta(s))H_-\big)\\
&\leq\theta(s)\kappa\big(H_+\big)+(1-\theta(s))\kappa\big(H_-\big)\\
&\leq\max\{\kappa(H_+),\kappa(H_-)\}\;.
\eea
We recall the abbreviation
\beq
\Delta_2=\max\{\kappa(H_+),\kappa(H_-)\}\;.
\eeq
Combining the previous two inequalities we obtain
\bea\label{eqn:estimate1}
|\eta(\sigma)|&\leq \frac{2}{2-\delta}\bigg(\max\{\A_-(w_-),-\A_+(w_+)\}+\int_0^1||H_-(\cdot,t)-H_+(\cdot,t)||_{C^0(M)}dt\\
&\qquad\qquad\qquad+2||\eta||_{C^0(\R)}||f_+-f_-||_{C^0(\Sigma)}+\delta/4+\Delta_2\bigg)\\
&\qquad+\frac{16}{\delta}\Big(\A_-(w_-)-\A_+(w_+)+||H_+-H_-||_-+2||\eta||_{C^0(\R)}||f_+-f_-||_{C^0(\Sigma)}\Big)\\[3ex]
&\leq\frac{32-14\delta}{\delta(2-\delta)}\Big(\int_0^1||H_-(\cdot,t)-H_+(\cdot,t)||_{C^0(M)}dt+2||\eta||_{C^0(\R)}||f_+-f_-||_{C^0(\Sigma)}\Big)\\
&\quad+\frac{2}{2-\delta}\max\{\A_-(w_-),-\A_+(w_+)\}+\frac{16}{\delta}\big(\A_-(w_-)-\A_+(w_+)\big)\\
&\quad+\frac{(\delta+4\Delta_2)}{4-2\delta}\;.
\eea
We recall the abbreviation
\beq
\Delta_0=\int_0^1||H_-(\cdot,t)-H_+(\cdot,t)||_{C^0(M)}dt\;.
\eeq
Since the right hand side of \eqref{eqn:estimate1} is independent of $\sigma$ we conclude that
\bea
||\eta||_{C^0(\R)}&\leq\frac{32-14\delta}{\delta(2-\delta)}\Big(\Delta_0+2||\eta||_{C^0(\R)}||f_+-f_-||_{C^0(\Sigma)}\Big) +\frac{(\delta+4\Delta_2)}{4-2\delta}\\
&\quad+\frac{2}{2-\delta}\max\{\A_-(w_-),-\A_+(w_+)\}+\frac{16}{\delta}\big(\A_-(w_-)-\A_+(w_+)\big)\\
\eea
thus
\bea
\Big(1-\frac{64-28\delta}{\delta(2-\delta)}||f_+-f_-||_{C^0(\Sigma)}\Big)||\eta||_{C^0(\R)}&\leq\frac{32-14\delta}{\delta(2-\delta)}\Delta_0+\frac{(\delta+4\Delta_2)}{4-2\delta}\\
&\quad+\frac{2}{2-\delta}\max\{\A_-(w_-),-\A_+(w_+)\}\\
&\quad+\frac{16}{\delta}\big(\A_-(w_-)-\A_+(w_+)\big)\;.
\eea
Now we recall our assumption
\beq
||f_+-f_-||_{C^0(\Sigma)}\leq\frac{\delta(2-\delta)}{128-56\delta}
\eeq
and therefore
\bea\label{eqn:inequality_2}
||\eta||_{C^0(\R)}&\leq\frac{64-28\delta}{\delta(2-\delta)}\Delta_0+\frac{(\delta+4\Delta_2)}{2-\delta}\\
&\quad+\frac{4}{2-\delta}\max\{\A_-(w_-),-\A_+(w_+)\}+\frac{32}{\delta}\big(\A_-(w_-)-\A_+(w_+)\big)\;.\\
\eea
Using the abbreviation
\beq
\Delta_1=||f_+-f_-||_{C^0(\Sigma)}
\eeq
and combining the inequalities \eqref{eqn:inequality_1} and \eqref{eqn:inequality_2} we obtain
\bea
\A_+(w_+)&\leq\A_-(w_-)+\Delta_0+2 \Delta_1\bigg(\frac{64-28\delta}{\delta(2-\delta)}\Delta_0+\frac{(\delta+4\Delta_2)}{2-\delta}\\
&\quad+\frac{4}{2-\delta}\max\{\A_-(w_-),-\A_+(w_+)\}+\frac{32}{\delta}\big(\A_-(w_-)-\A_+(w_+)\big)\bigg)\;.
\eea
In the case $\A_+(w_+)\leq\A_-(w_-)$ or  $0\geq\A_+(w_+)$ the assertion of the Proposition to be proved follows trivially. Therefore, from now on we assume that $\A_+(w_+)\geq\A_-(w_-)$ and $\A_+(w_+)\geq0$. Then we can simplify the above estimate to
\bea
\A_+(w_+)&\leq\A_-(w_-)+\Delta_0\\
&\quad+2\Delta_1\bigg(\frac{64-28\delta}{\delta(2-\delta)}\Delta_0+\frac{(\delta+4\Delta_2)}{2-\delta}+\frac{4}{2-\delta}\max\{\A_-(w_-),-\A_+(w_+)\}\bigg)\;.
\eea
Next we distinguish two cases. If $\A_+(w_+)\geq\A_-(w_-)\geq0$ then
\bea
\A_+(w_+)&\leq\A_-(w_-)+\Delta_0+2\Delta_1\bigg(\frac{64-28\delta}{\delta(2-\delta)}\Delta_0+\frac{(\delta+4\Delta_2)}{2-\delta}+\frac{4}{2-\delta}\A_-(w_-)\bigg)\\
&=\A_-(w_-)\Big(1+\frac{8\Delta_1}{2-\delta}\Big)+\Delta_0+2\Delta_1\bigg(\frac{64-28\delta}{\delta(2-\delta)}\Delta_0+\frac{(\delta+4\Delta_2)}{2-\delta}\bigg)\;.
\eea
If $\A_+(w_+)\geq0\geq\A_-(w_-)$ then
\bea
\A_+(w_+)&\leq\Delta_0+2\Delta_1\bigg(\frac{64-28\delta}{\delta(2-\delta)}\Delta_0+\frac{(\delta+4\Delta_2)}{2-\delta}\bigg)\;.
\eea
The Proposition follows from the last two inequalities.
\end{proof}

\begin{Thm}\label{thm:main_action_estimate}
Let $\Mp_\pm=(F_{f_\pm},H_\pm)\in\MP^{reg}(\Sigma)$. We abbreviate $\Sigma_\pm:=\Sigma_{f_\pm}$. Then
\beq\label{eqn:crucial_inequ_spectral_intro}
\sigma_{\Mp_-}(X_-)\leq e^{\frac{16\Delta_1}{2-\delta}}\max \big\{\sigma_{\Mp_-}(X_-),0\big\}+\Big(\frac{(2-\delta)\Delta_0}{\Delta_1}+2(\delta+4 \Delta_2)\Big)\Big(e^{\frac{16\Delta_1}{2-\delta}}-1\Big)
\eeq 
where $X_+=\zeta_{\Mp_-}^{\Mp_+}(X_-)\neq0\in\RFH_*(\Mp_+)$ and $\Delta_0$, $\Delta_1$, and $\Delta_2$ are as in Proposition \ref{prop:main_action_estimate}.
\end{Thm}

\begin{proof}
Under the assumption 
\beq
||f_+-f_-||_{C^0(\Sigma)}\leq\frac{\delta(2-\delta)}{128-56\delta}
\eeq
Proposition \ref{prop:main_action_estimate} implies as in Corollary \ref{cor:action_estimate_continuation_Morse}  that
\beq
\sigma_{\Mp_+}(X_+)\leq\max\Big\{\Big(1+\frac{8\Delta_1}{2-\delta}\Big)\sigma_{\Mp_-}(X_-),0\Big\}+\Delta_0+2\Delta_1\bigg(\frac{64-28\delta}{\delta(2-\delta)}\Delta_0+\frac{(\delta+4\Delta_2)}{2-\delta}\bigg)\;.
\eeq
where $X_1=\zeta_{\Mp_0}^{\Mp_1}(X_0)\neq0\in\RFH_*(\Mp_1)$. In general this assumption is not satisfied. But we can always split the homotopy from $f_-$ to $f_+$ into many small homotopies each of which satisfies the above inequality. To obtain the statement of the theorem we eventually take an adiabatic limit. We again define
\beq
f_s:=\theta(s)(f_+-f_-)+f_-\,,\quad s\in\R
\eeq
and
\beq
F_s:=F_{f_s}\quad\text{and}\quad H_s:=\theta(s)(H_+-H_-)+H_-\;.
\eeq
where $\theta$ is the cut-off function defined above Proposition \ref{prop:main_action_estimate}. We choose $N\in\N$ such that
\beq
N\geq\frac{256-112\delta}{\delta(2-\delta)}||f_+-f_-||_{C^0(\Sigma)}
\eeq
and set for $k=0,\ldots,N$
\beq
f^k:=f_{\frac{k}{N}}\,,\qquad H^k:=H_{\frac{k}{N}},\quad\text{and}\quad\Mp^k:=(F_{f^k},H^k)\;.
\eeq
For convenience we proceed with the proof under the assumption that $\Mp^k$ is a regular Moser pair. Otherwise, in the following arguments $\Mp^k$ has to be replaced by an arbitrarily small regular perturbation. By taking the limit this does not influence  the action estimates. We recall that $0\leq\theta'\leq2$ and observe that by the choice of $N$  
\bea
||f^{k+1}-f^k||_{C^0(\Sigma)}&=||\Big(\theta(\tfrac{k+1}{N})-\theta(\tfrac{k}{N})\Big)(f_+-f_-)||_{C^0(\Sigma)}\\
&\leq 2\big(\tfrac{k+1}{N}-\tfrac{k}{N}\big)||f_+-f_-||_{C^0(\Sigma)}\\
&\leq \tfrac{2}{N}||f_+-f_-||_{C^0(\Sigma)}\\
&\leq\frac{\delta(2-\delta)}{128-56\delta}\;.
\eea
In particular,
\beq
\Delta_1^k:=||f^{k+1}-f^k||_{C^0(\Sigma)}\leq\tfrac{2}{N}\Delta_1\;.
\eeq
Similarly,
\beq
||H_{\frac{k+1}{N}}(\cdot,t)-H_{\frac{k}{N}}(\cdot,t)||_{C^0(M)}\leq\tfrac{2}{N}||H_{+}(\cdot,t)-H_{-}(\cdot,t)||_{C^0(M)}
\eeq
and therefore
\beq
\Delta_0^k:=\int_0^1||H_{\frac{k+1}{N}}(\cdot,t)-H_{\frac{k}{N}}(\cdot,t)||_{C^0(M)}dt\leq\tfrac{2}{N}\Delta_0\;.
\eeq
Finally, since $\kappa$ is a semi-norm
\beq
\kappa(H^k)\leq\max\{\kappa(H_+),\kappa(H_-)\}=\Delta_2\;.
\eeq
Thus, we conclude from Proposition \ref{prop:main_action_estimate} as explained at the beginning of the proof
\bea
\sigma_{\Mp^{k+1}}(X^{k+1})&\leq\max\Big\{\Big(1+\frac{8\Delta_1^k}{2-\delta}\Big)\sigma_{\Mp^{k}}(X^{k}),0\Big\}+\Delta_0^k+2\Delta_1^k\bigg(\frac{64-28\delta}{\delta(2-\delta)}\Delta_0^k+\frac{(\delta+4 \Delta_2)}{2-\delta}\bigg)\\
&\leq\max\Big\{\Big(1+\frac{16\Delta_1}{2-\delta}\frac{1}{N}\Big)\sigma_{\Mp^{k}}(X^{k}),0\Big\}+{\frac{2}{N}\Delta_0}+{\frac{4}{N}\Delta_1}\bigg(\frac{64-28\delta}{\delta(2-\delta)}{\frac{2}{N}\Delta_0}+\frac{(\delta+4 \Delta_2)}{2-\delta}\bigg)\\
\eea
where $X^{k+1}=\zeta_{\Mp^{k}}^{\Mp^{k+1}}(X^{k})=\zeta_{\Mp^k}^{\Mp^{k+1}}(X_-)$. Lemma \ref{lemma:iteration_formula_estimate} implies that
\bean
\sigma_{\Mp_+}(X_+)&\leq \Big(1+\frac{16\Delta_1}{2-\delta}\frac{1}{N}\Big)^N\max \bigg\{\sigma_{\Mp_-}(X_-),{\frac{2}{N}\Delta_0}+{\frac{4}{N}\Delta_1}\bigg(\frac{64-28\delta}{\delta(2-\delta)}{\frac{2}{N}\Delta_0}+\frac{(\delta+4 \Delta_2)}{2-\delta} \bigg)\bigg\}\\
&\quad+\bigg({\frac{2}{N}\Delta_0}+{\frac{4}{N}\Delta_1} \bigg(\frac{64-28\delta}{\delta(2-\delta)}{\frac{2}{N}\Delta_0}+\frac{(\delta+4 \Delta_2)}{2-\delta} \bigg)\bigg)\frac{\Big(1+\frac{16\Delta_1}{2-\delta}\frac{1}{N}\Big)^N-1}{\Big(1+\frac{16\Delta_1}{2-\delta}\frac{1}{N}\Big)-1}\\
&= \Big(1+\frac{16\Delta_1}{2-\delta}\frac{1}{N}\Big)^N\max \bigg\{\sigma_{\Mp_-}(X_-),{\frac{2}{N}\Delta_0}+{\frac{4}{N}\Delta_1}\bigg(\frac{64-28\delta}{\delta(2-\delta)}{\frac{2}{N}\Delta_0}+\frac{(\delta+4 \Delta_2)}{2-\delta} \bigg)\bigg\}\\
&\quad+N\frac{2-\delta}{16\Delta_1}\bigg({\frac{2}{N}\Delta_0}+{\frac{4}{N}\Delta_1} \bigg(\frac{64-28\delta}{\delta(2-\delta)}{\frac{2}{N}\Delta_0}+\frac{(\delta+4 \Delta_2)}{2-\delta} \bigg)\bigg)\bigg(\Big(1+\frac{16\Delta_1}{2-\delta}\frac{1}{N}\Big)^N-1\bigg)\\
&= \Big(1+\frac{16\Delta_1}{2-\delta}\frac{1}{N}\Big)^N\max \bigg\{\sigma_{\Mp_-}(X_-),{\frac{2}{N}\Delta_0}+{\frac{4}{N}\Delta_1}\bigg(\frac{64-28\delta}{\delta(2-\delta)}{\frac{2}{N}\Delta_0}+\frac{(\delta+4 \Delta_2)}{2-\delta}\bigg) \bigg\}\\
&\quad+\frac{2-\delta}{16\Delta_1}\bigg({2\Delta_0}+4{\Delta_1} \bigg(\frac{64-28\delta}{\delta(2-\delta)}{\frac{2}{N}\Delta_0}+\frac{(\delta+4 \Delta_2)}{2-\delta} \bigg)\bigg)\bigg(\Big(1+\frac{16\Delta_1}{2-\delta}\frac{1}{N}\Big)^N-1\bigg)\;.
\eea
In the limit $N\to\infty$
\bea
\sigma_{\Mp_+}(X_+)&\leq e^{\frac{16\Delta_1}{2-\delta}}\max \big\{\sigma_{\Mp_-}(X_-),0\big\}+\frac{2-\delta}{8\Delta_1}\bigg({\Delta_0}+2{\Delta_1}\frac{(\delta+4 \Delta_2)}{2-\delta}\bigg)\bigg(e^{\frac{16\Delta_1}{2-\delta}}-1\bigg)\\
&= e^{\frac{16\Delta_1}{2-\delta}}\max \big\{\sigma_{\Mp_-}(X_-),0\big\}+\frac18\Big(\frac{(2-\delta)\Delta_0}{\Delta_1}+2(\delta+4 \Delta_2)\Big)\Big(e^{\frac{16\Delta_1}{2-\delta}}-1\Big)\;.
\eea
\end{proof}

\begin{Def}\label{def:norm_on_adapted_Moser_pairs}
We define the norm of an adapted Moser pair $\Mp=(F_f,H)\in\MP(\Sigma)$ by
\beq
||\Mp||:=||f||_{C^0(\Sigma)}+\int_0^1||H(\cdot,t)||_{C^0(M)}dt+\kappa(H)\;.
\eeq
We denote by
\beq
\mathcal{D}(\overline{\Mp}):=\{\Mp'=(F',H')\mid ||\overline{\Mp}-\Mp'||<1\}
\eeq
the open 1-ball around $\overline{\Mp}$ in $\MP(\Sigma)$.
\end{Def}
Estimating $\Delta_0,\Delta_1\leq||\Mp_+-\Mp_-||$ and $\Delta_2\leq \max\{||\Mp_+||,||\Mp_-||\} $ and using of the monotonicity of $x\mapsto\frac{e^x-1}{x}$ for $x\geq0$ we immediately obtain from Theorem \ref{thm:main_action_estimate} the following corollary.
\begin{Cor}\label{cor:spectral_inequality_implied_by_main_thm}
Under the assumptions of Theorem \ref{thm:main_action_estimate} we have
\beq
\sigma_{\Mp_+}(X_+)\leq e^{\frac{16||\Mp_+-\Mp_-||}{2-\delta}}\max \big\{\sigma_{\Mp_-}(X_-),0\big\}+\frac18\Big(2+\delta+8\max\{||\Mp_+||,||\Mp_-||\}\Big)\Big(e^{\frac{16||\Mp_+-\Mp_-||}{2-\delta}}-1\Big)
\eeq 
\end{Cor}

\begin{Def}
For a weakly regular Moser pair $\Mp\in\MP^{reg}(\Sigma)$  and $X\neq0\in\RFH_*$ we set
\beq
\sigma_\Mp(X):=\sigma\Big(\zeta^{\Mp}(X)\Big)
\eeq
where $\RFH_*$ and $\zeta^\Mp$ are defined in Definition \ref{def:zeta_RFH}.
Moreover, for fixed $X\neq0\in\RFH_*$ we define
\bea
\rho_X:\MP^{reg}(\Sigma)&\pf\R\\
\Mp&\mapsto \sigma_\Mp(X)\;.
\eea
\end{Def}

\begin{Conv}
From now on we fix a weakly regular Moser pair $\Mp_0=(F_{f_0},H_0)\in\MP^{reg}(\Sigma)$.
\end{Conv}

\begin{Lemma}\label{lemma:local_lipschitz_continuity}
For a Moser pair $\overline{\Mp} =(\overline{F}, \overline{H})$ we define
\beq
\B(\overline{\Mp}):=\bigg\{X\in\RFH_*\mid \sigma_{\Mp_0}(X)>\frac18\Big(2+\delta+8\max \big\{||\Mp_0||,||\overline{\Mp}||+1 \big\}\Big)\Big(e^{\frac{16(||\Mp_0-\overline{\Mp}||+1)}{2-\delta}}-1\Big)\bigg\}
\eeq
If $X\in\B(\overline{\Mp})$ then the map 
\bea
\rho_X:\MP^{reg}\pf\R,\qquad\Mp\mapsto \sigma_\Mp(X)
\eea
is locally Lipschitz continuous around $\overline{\Mp}$ with respect to the norm on $\MP(\Sigma)$ introduced in Definition \ref{def:norm_on_adapted_Moser_pairs}.
\end{Lemma}

\begin{proof}
We recall that $\mathcal{D}(\overline{\Mp})$ denotes the open 1-ball around $\overline{\Mp}$ in $\MP(\Sigma)$. We assume by contradiction $\sigma_{\Mp'}(X)\leq0$, $\forall \Mp'\in\mathcal{D}(\overline{\Mp})\cap\Mp^{reg}$. Then applying Corollary \ref{cor:spectral_inequality_implied_by_main_thm} to $\sigma_{\Mp_-}(X)\equiv\sigma_{\Mp'}(X)$ and $\sigma_{\Mp_+}(X)\equiv\sigma_{\Mp_0}(X)$ we obtain
\bea
\sigma_{\Mp_0}(X)\leq\frac18\Big(2+\delta+8 \max\big\{||\Mp_0||,||\Mp'||\big\})\Big)\Big(e^{\frac{16 ||\Mp_0-\Mp'||}{2-\delta}}-1\Big)
\eea
From
\beq
||\Mp'||<||\overline{\Mp}||+1,\quad||\Mp_0-\Mp'||<||\Mp_0-\overline{\Mp}||+1
\eeq
we get
\bea
\sigma_{\Mp_0}(X)&\leq\frac18\Big(2+\delta+8 \max\big\{||\Mp_0||,||\Mp'||\big\})\Big)\Big(e^{\frac{16 ||\Mp_0-\Mp'||}{2-\delta}}-1\Big)\\
&\leq\frac18\Big(2+\delta+8 \max\big\{||\Mp_0||,||\overline{\Mp}||+1\big\})\Big)\Big(e^{\frac{16(||\Mp_0-\overline{\Mp}||+1)}{2-\delta}}-1\Big)\\
\eea
This contradicts the assumption that $X\in\B(\overline{\Mp})$. Thus, we conclude 
\beq\label{eqn:inequality3}
\sigma_{\Mp'}(X)\geq0\qquad\forall \Mp'\in\mathcal{D}(\overline{\Mp})\cap\Mp^{reg}\;.
\eeq
We choose $\Mp',\Mp''\in\mathcal{D}(\overline{\Mp})\cap\Mp^{reg}$ and estimate for $X\in\B(\overline{\Mp})$ using again Corollary \ref{cor:spectral_inequality_implied_by_main_thm} and $\Mp',\Mp''\in\mathcal{D}(\overline{\Mp})$ and employing the elementary estimate $e^t\leq 1+e^C\,t$ or $e^t-1\leq e^C\,t$, $\forall t\in[0,C]$:
\bean
\rho_X(\Mp')&\leq e^{\frac{16||\Mp'-\Mp''||}{2-\delta}}\rho_{X}(\Mp'')+\frac18\Big(2+\delta+8 \max\big\{||\Mp'||,||\Mp''||\big\})\Big)\Big(e^{\frac{16 ||\Mp'-\Mp''||}{2-\delta}}-1\Big)\\
&\leq \Big(1+e^{\frac{32}{2-\delta}}\frac{16}{2-\delta}||\Mp'-\Mp''||\Big)\rho_{X}(\Mp'')\\
&\quad+\frac18\Big(10+\delta+8||\overline{\Mp}||)\Big)\Big(e^{\frac{32}{2-\delta}}\frac{16}{2-\delta}||\Mp'-\Mp''||\Big)\\
&=\rho_{X}(\Mp'')+\Big(e^{\frac{32}{2-\delta}}\frac{16}{2-\delta}||\Mp'-\Mp''||\Big)\rho_{X}(\Mp'')\\
&\quad+\frac18\Big(10+\delta+8||\overline{\Mp}||)\Big)\Big(e^{\frac{32}{2-\delta}}\frac{16}{2-\delta}||\Mp'-\Mp''||\Big)\;.
\eea
Using again Corollary \ref{cor:spectral_inequality_implied_by_main_thm} we estimate
\bea
\rho_X(\Mp'')&\leq e^{\frac{16||\Mp''-\Mp_0||}{2-\delta}}\max \big\{\rho_X(\Mp_0),0\big\}\\
&\quad+\frac18\Big(2+\delta+8\max\{||\Mp''||,||\Mp_0||\}\Big)\Big(e^{\frac{16|| \Mp''-\Mp_0 ||}{2-\delta}}-1\Big)\\
&\leq e^{\frac{16(||\overline{\Mp}-\Mp_0||+1)}{2-\delta}}\max \big\{\rho_X(\Mp_0),0\big\}\\
&\quad+\frac18\Big(2+\delta+8\max\{{||\overline{\Mp}||+1},||\Mp_0||\}\Big)\Big(e^{\frac{16(||\overline{\Mp}-\Mp_0||+1)}{2-\delta}}-1\Big)\\
&=:C(\Mp_0,X)\;.
\eea
Combining the last two inequalities we see
\bean
\rho_X(\Mp')&\leq\rho_{X}(\Mp'')+C(\Mp_0,X)e^{\frac{32}{2-\delta}}\frac{16}{2-\delta}||\Mp'-\Mp''||\\
&\quad+\frac18\Big(10+\delta+8||\overline{\Mp}||)\Big)\Big(e^{\frac{32}{2-\delta}}\frac{16}{2-\delta}||\Mp'-\Mp''||\Big)
\eea
and thus
\bea
\rho_X(\Mp'')-\rho_X(\Mp')&\leq C(\Mp_0,X)e^{\frac{32}{2-\delta}}\frac{16}{2-\delta}||\Mp'-\Mp''||\\
&\quad+\frac18\Big(10+\delta+8||\overline{\Mp}||)\Big)\Big(e^{\frac{32}{2-\delta}}\frac{16}{2-\delta}||\Mp'-\Mp''||\Big)\\
&\leq D(\overline{\Mp},\Mp_0,X)||\Mp'-\Mp''||
\eea
where we abbriviate
\beq
D(\overline{\Mp},\Mp_0,X):=C(\Mp_0,X)e^{\frac{32}{2-\delta}}\frac{16}{2-\delta}+\frac18\Big(10+\delta+8||\overline{\Mp}||)\Big)\Big(e^{\frac{32}{2-\delta}}\frac{16}{2-\delta}\Big)\;.
\eeq
By symmetry
\bea
|\rho_X(\Mp'')-\rho_X(\Mp')|&\leq D(\overline{\Mp},\Mp_0,X)||\Mp'-\Mp''||\;.
\eea
This proves the Lemma.
\end{proof}

We recall that we fixed $\Mp_0\in\MP(\Sigma)$. Similarly as in Corollary \ref{cor:spectral_invariants_extend_to_C_0_Morse} we can extend $\rho_X$.
\begin{Cor}\label{cor:existence_of_extension_of_rho}
Let $\overline{\Mp}\in\Mp(\Sigma)$ and $X\in\B(\overline{\Mp})$. Then there exists a Lipschitz continuous function
\beq
\overline{\rho}_X:\mathcal{D}(\overline{\Mp})\pf(0,\infty)
\eeq
satisfying
\beq
\overline{\rho}_X(\Mp')={\rho}_X(\Mp')\qquad\forall\Mp'\in \mathcal{D}(\overline{\Mp})\cap\MP^{reg}(\Sigma)\;.
\eeq
Moreover, is a spectral value
\beq
\overline{\rho}_X(\Mp)\in\S(\A^\Mp)\;.
\eeq
Finally, we have
\beq
\rho_X(\Mp_0)\leq e^{\frac{16||\Mp_0-\overline{\Mp}||}{2-\delta}}\overline{\rho}_X(\overline{\Mp})+\frac18\Big(2+\delta+8\max\{||\Mp_0||,||\overline{\Mp}||\}\Big)\Big(e^{\frac{16||\Mp_0-\overline{\Mp}||}{2-\delta}}-1\Big)\;.
\eeq
\end{Cor}

\begin{proof}
That $\rho_X$ has an extension as a Lipschitz continuous function follows immediately from Lemma \ref{lemma:local_lipschitz_continuity} and the fact that $\Mp^{reg}(\Sigma)$ is dense in $\Mp(\Sigma)$, see Proposition \ref{prop:generic_is_regular}.

To prove that $\overline{\rho}_X(\Mp)$ is a critical value of $\A^\Mp$ we choose a sequence $\Mp_n\in\Mp^{reg}(\Sigma)\cap\mathcal{D}(\overline{\Mp})$ with $\Mp_n\pf\Mp$. Then by Lemma \ref{lemma:finite_spectral_value_implies_critical_value} there exist $w_n=(v_n,\eta_n)\in\Crit\A^{\Mp_n}$ with
\beq
\rho_X(\Mp_n)=\A^{\Mp_n}(w_n)\;.
\eeq
Moreover, by Lemma \ref{lem:crucial_lemma_for_eta_bound} we conclude
\bea
|\eta_n|&\leq C\big(|\A^{\Mp_n}(w_n)|+1\big)\\
&= C\big(|\rho_X(\Mp_n)|+1\big)\\
&\leq C\big(|\overline{\rho}_X(\overline{\Mp})|+D(\overline{\Mp},\Mp_0,X)||\overline{\Mp}-\Mp_n||+1)\\
&\leq C\big(|\overline{\rho}_X(\overline{\Mp})|+D(\overline{\Mp},\Mp_0,X)+1)
\eea
by Lipschitz continuity and definition of $\mathcal{\D}(\overline{\Mp})$. In particular, the sequence $\eta_n$ is uniformly bounded and applying the Theorem of Arzela-Ascoli $w_{n_\nu}\to w^*\in\Crit\A^\Mp$ and 
\beq
\overline{\rho}_X(\Mp)=\A^\Mp(w^*)\;.
\eeq
The last inequality claimed in the statement of the Corollary follows from Corollary \ref{cor:spectral_inequality_implied_by_main_thm} together with the observation that   $\overline{\rho}_X(\Mp)\geq0$. The latter follows from \eqref{eqn:inequality3} by continuity of $\rho_X$.
\end{proof}

\begin{Def}
For an adapted Moser pair $\Mp\in\MP(\Sigma)$ and $X\in\B(\overline{\Mp})$ we define
\beq
\overline{\sigma}_{\overline{\Mp}}(X):=\overline{\rho}_X(\overline{\Mp})\;.
\eeq 
\end{Def}

\begin{Cor}\label{cor:last_cor}
We recall that we fixed a weakly regular $\Mp_0$. If
\beq
\{\sigma_{\Mp_0}(X)\mid X\in\RFH_*\}\subset \R\cup\{-\infty\}
\eeq 
is unbounded from above then
\beq
\{\overline{\sigma}_{\overline{\Mp}}(X)\mid X\in\B(\overline{\Mp})\}\subset (0,\infty)
\eeq 
is also unbounded from above for all $\overline{\Mp}\in\MP(\Sigma)$.
\end{Cor}

\begin{proof}
The assumption that the spectral values are unbounded together with the definition of $\B(\overline{\Mp})$, see Lemma \ref{lemma:local_lipschitz_continuity}, implies that also the set
\beq
\{\sigma_{\Mp_0}(X)\mid X\in\B(\overline{\Mp})\}\subset (0,\infty)
\eeq 
is unbounded from above. Combining this with the estimate in Corollary \ref{cor:existence_of_extension_of_rho} implies the assertion.
\end{proof}

\section{Proof of Theorem \ref{thm:main_intro}}

We recall that in Theorem \ref{thm:main_intro} we assume that $(M=T^*B,\om)$ where $B$ is a closed manifold and $S\subset M$ is fiber-wise star-shaped hypersurface. We fix a a bumpy metric $g$ in the sense of Abraham \cite{Abraham_Bumpy_metrics} and set
\beq
\Sigma:=\{(q,p)\in T^*B\mid ||p||_g^2=1\}\;.
\eeq
According to the Theorem of Abraham \cite{Abraham_Bumpy_metrics} bumpy metrics exist (and are even dense). Since $g$ is bumpy the Moser pair
\beq
\Mp_0:=(F_{f_0},0)\in\Mp(\Sigma)
\eeq
is weakly-regular if we choose $f_0=0$. A hypersurface in $T^*B$ is fiber-wise star-shaped if and only if it is of the form $\Sigma_f$ for some $f:\Sigma\pf\R$. In particular, there exists a function $f_S:\Sigma\pf\R$ with
\beq
S=\Sigma_{f_S}\;.
\eeq

\begin{Prop}\label{prop:pos_CZ_implies_pos_spectral_value}
With the above notation we have
\beq
\CZ(X)\geq0\quad\Longrightarrow\quad\sigma_{\Mp_0}(X)\geq0\qquad\forall X\in\RFH_*({\Mp_0})\;.
\eeq
\end{Prop}

\begin{proof}
First we recall that critical points $(u,\eta)$ of $\A^{\Mp_0}$ with positive/negative $\eta$ are positively/negatively parametrized geodesics for $g$ and that the Conley-Zehnder index coincides with the negative of the Morse index. In particular, positive Conley-Zehnder index implies negatively parametrised geodesics. Let 
\beq
\xi=\sum_{c\colon\CZ(c)=k}\xi_c\,c\in\RFC_{\geq0}(\Mp_0)
\eeq
then since ${\Mp_0}=(F_{f_0},0)$ the action value $\A^{\Mp_0}(u,\eta)=-\eta$ is the negative of the period of the geodesic. In particular, $\A^{\Mp_0}(c)\geq0$ if $\CZ(c)\geq0$.
\end{proof}

\begin{Lemma}\label{lemma:bounded_spectral_value_implies_finite_number}
Under the same assumptions as in Theorem \ref{thm:main_intro} for each $\kappa>0$ the set
\beq
\mathscr{R}_\kappa:=\{X\in\RFH_{\geq0}\mid0\leq\sigma_{\Mp_0}(X)\leq \kappa\}
\eeq
is finite.
\end{Lemma}

\begin{proof}
We fix an auxiliary Morse function $f$ on the critical set $\Crit\A^{\Mp_0}$. Then
\beq
\mathscr{C}_\kappa:=\{c\in\Crit(f)\mid0\leq\A^{\Mp_0}(c)\leq\kappa\}
\eeq
is finite, see Remark \ref{rmk:perturb_by_Morse_fctn} for notation. Indeed, this follows from the theorem of Arzela-Ascoli together with assumption that ${\Mp_0}$ is weakly regular, see also the proof of Lemma \ref{lemma:finite_spectral_value_implies_critical_value}. If $X\in\RFH_{\geq0}({\Mp_0})$ and $\sigma_{\Mp_0}(X)\leq \kappa $ then $X$ is of the form
\beq
X=\sum_{c\in\mathscr{C}_\kappa}\xi_cc
\eeq
with $\xi_c\in\Z/2$, and therefore,
\beq
\#\mathscr{R}_\kappa\leq2^{\#\mathscr{C}_\kappa}
\eeq
is finite.
\end{proof}

\begin{Prop}\label{prop:Liebling_has_arb_large_crtitical_values}
Under the same assumptions as in Theorem \ref{thm:main_intro} the set
\beq
\{\sigma_{\Mp_0}(X)\mid X\in\RFH_{\geq0}\}
\eeq
is unbounded from above.
\end{Prop}

\begin{proof}
Assume by contradiction that there exists $\kappa>0$ such that
\beq
\sigma_{\Mp_0}(X)\leq \kappa
\eeq
for all $X\in\RFH_{\geq0}({\Mp_0})$. From the Proposition \ref{prop:pos_CZ_implies_pos_spectral_value} we also know $0\leq\sigma_{\Mp_0}(X)$. We recall from \cite{Cieliebak_Frauenfelder_Oancea_Rabinowitz_Floer_homology_and_symplectic_homology,Abbondandolo_Schwarz_Estimates_and_computations_in_Rabinowitz_Floer_homology} that the assumption on $\H_*(\L_B)$ implies the same for Rabinowitz Floer homology, that is, 
\beq
\dim\RFH_*({\Mp_0})=\infty\;.
\eeq
Thus, the set
\beq
\mathscr{R}_\kappa =\{X\in\RFH_{\geq0}\mid0\leq\sigma_{\Mp_0}(X)\leq \kappa\}
\eeq
is infinite. This directly contradicts Lemma \ref{lemma:bounded_spectral_value_implies_finite_number}.
\end{proof}

To finish the proof of Theorem \ref{thm:main_intro} we set $\Mp_S=(F_{f_S},\widehat{H})$ where $f_S$ is as above and $\widehat{H}\in\mathcal{H}$ is such that $\phi_{\widehat{H}}^1=\phi_H^1$. We apply Corollary \ref{cor:last_cor} to $\Mp_0$ and conclude that
\beq
\{\overline{\sigma}_{\Mp_S}(X)\mid X\in\B(\Mp_S)\}\subset (0,\infty)
\eeq 
is unbounded from above. Thus, $\A^{\Mp_S}$ has arbitrarily large critical values. At a critical point $(v,\eta)\in\Crit\A^{\Mp_S}$ we compute
\beq
\A^{\Mp_S}(v,\eta)=-\eta-\int\big[\lambda(X_{\widehat{H}}(v(t),t))+\widehat{H}(t,v(t))\big]dt
\eeq
and thus
\beq
\eta\leq-\A^{\Mp_S}(v,\eta)+\kappa(\widehat{H})
\eeq
where $\kappa(\widehat{H})$ is the seminorm defined in Definition \ref{def:constant_kappa(H)}. In particular, there exist critical points of $\A^{\Mp_S}$ with arbitrarily negative $\eta$-value. This proves Theorem \ref{thm:main_intro} for negative $\eta$-values.

Looking at Rabinowitz Floer co-homology the statement for positive $\eta$-values follows.

\appendix
\section{An iteration inequality}

Let $x_n$, $n\geq0$ be a sequence of numbers satisfying
\beq
x_{n+1}\leq \max\{\alpha x_n,0\}+\beta
\eeq
for numbers $\alpha>0$, and $\beta>0$.

\begin{Lemma}\label{lemma:iteration_formula_estimate}
\beq
x_n\leq \alpha^n\max\{x_0,\beta\}+\beta\sum_{j=0}^{n-1}\alpha^j=\alpha^n\max\{x_0,\beta\}+\beta\;\frac{\alpha^n-1}{\alpha-1}
\eeq
\end{Lemma}

\begin{proof}
The proof goes by induction on $n$. For $n=0$ we check
\beq
x_0\leq \max\{x_0,\beta\}=\alpha^0\max\{x_0,\beta\}+\beta\sum_{j=0}^{-1}\alpha^j\;.
\eeq
For the induction step $n\to n+1$ we distinguish two cases. 

Case 1: $x_n\leq0$. Then
\bea
x_{n+1}&\leq\max\{\alpha x_n,0\}+\beta\\[3ex]
&=\beta\\
&\leq \underbrace{\alpha^{n+1}\max\{x_0,\beta\}}_{\geq0}+\beta\cdot\underbrace{\sum_{j=0}^{n}\alpha^j}_{\geq1}\;.
\eea

Case 2: $x_n>0$. Then
\bea
x_{n+1}&\leq\max\{\alpha x_n,0\}+\beta\\
&\leq \alpha x_n+\beta\\
&\leq \alpha \Big(\alpha^n\max\{x_0,\beta\}+\beta\sum_{j=0}^{n-1}\alpha^j\Big)+\beta\\
&=\alpha^{n+1}\max\{x_0,\beta\}+\beta\sum_{j=0}^{n}\alpha^j
\eea
where we used the induction hypothesis in the third inequality. This proves the Lemma.
\end{proof}

%
%
\bibliographystyle{amsalpha}
\bibliography{../../../Bibtex/bibtex_paper_list}
\end{document}